\documentclass[reqno]{amsart}
\usepackage{amssymb,amsmath,amsthm,mathrsfs,multirow,enumerate,mathdots}
\usepackage{color}
\usepackage{graphicx} 
\usepackage[all]{xy}
\usepackage{enumitem}
\usepackage{breqn}
\usepackage{orcidlink}
\usepackage{tikz}
\usepackage[capitalize]{cleveref}
\usepackage{dsfont}
\usepackage{subcaption}

\numberwithin{equation}{section}
\newtheorem{thm}{Theorem}[section]
\newtheorem{lem}[thm]{Lemma}
\newtheorem{cor}[thm]{Corollary}
\newtheorem{prop}[thm]{Proposition}
\theoremstyle{definition}
\newtheorem{ex}[thm]{Example}

\newtheorem{rmk}[thm]{Remark}

\theoremstyle{definition} \numberwithin{equation}{section}

\begin{document}
\title[Quantum fractional revival on zero-divisor graphs over
$\mathbb{Z}_n$]{Quantum fractional revival on zero-divisor graphs over $\mathbb{Z}_n$}

\author{Bui Phuoc Minh$^{1,2}$ \orcidlink{0000-0002-8876-5766} and Songpon
Sriwongsa$^*$ \orcidlink{0000-0002-5137-8113}}

\thanks{*Corresponding Author}

\address{
	Bui Phuoc Minh \\
	$^{1}$Department of Mathematics, Faculty of Science, Mahidol University,
	Bangkok 10400, Thailand \\
	$^{2}$Centre of Excellence in Mathematics, MHESI, Bangkok 10400, Thailand
}
\email{\tt minhbui.phu@mahidol.ac.th}

\address{Songpon Sriwongsa \\ Department of Mathematics \\ Faculty of Science
	\\ King Mongkut's University of Technology Thonburi (KMUTT) \\ 126 Pracha Uthit Rd., Bang Mod, Thung Khru \\ Bangkok 10140, Thailand}
\email{\tt  songponsriwongsa@gmail.com}



\begin{abstract}
In this paper, we characterize the existence of perfect state transfer (PST) and
fractional revival in continuous-time quantum walks on the zero-divisor graph
$\Gamma(\mathbb{Z}_n)$. By using the canonical equitable partition of
$\Gamma(\mathbb{Z}_n)$ induced by the proper divisors of $n$, we derive a sufficient
condition on $n$ for PST to occur between a pair of vertices. We show that 
fractional revival is restricted to cells of size $2$ within the equitable
partition. Furthermore, assuming $-1$ is not an eigenvalue of the quotient
spectrum, we establish that two vertices in $\Gamma(\mathbb{Z}_n)$ are strongly
cospectral if and only if they form a cell of size $2$ within the equitable
partition that is either a set of false twins or true twins. Finally, we provide
a characterization of fractional revival on bipartite $\Gamma(\mathbb{Z}_n)$ and
prove the non-existence of fractional revival on $\Gamma(\mathbb{Z}_{p^2q})$. 
\end{abstract}

\date{}

\maketitle

\section{Introduction}
Let $G = (V, E)$ be a finite, undirected, simple graph with the adjacency matrix $A(G)$.
The \emph{transition matrix} $H(t)$ associated with $A(G)$ is defined as
\[
	H(t) := \exp(itA(G)) = \sum_{k \ge 0} \frac{(it)^k}{k!} (A(G))^k, \quad t \ge 0,
\]
where $i=\sqrt{-1}$.
Since $A(G)$ is real and symmetric, $H(t)$ is also symmetric and unitary.
In particular,
\[
	\overline{H(t)}=H(t)^{-1},
\]
where $\overline{\ \cdot\ }$ denotes complex conjugation, and hence
\[
	(\overline{H(t)})^{T}=H(t)^{-1}.
\]
The matrix $H(t)$ governs a continuous-time quantum walk on $G$,
a fundamental concept in quantum computation and quantum information theory.
Quantum walks serve as algorithmic tools in various contexts, including
Grover-type search procedures \cite{Grover1996Fast} and the decision-tree algorithm of Farhi
and Gutmann \cite{Farhi1998Quantum}.
Moreover, Bose \cite{Bose2003Quantum} proposed quantum walks on graphs as models for quantum state transfer in quantum communication networks.

\medskip

We say that $(\alpha, \beta)$-\emph{fractional revival} (FR) occurs from a
vertex $u$ to a vertex $v$ at time $\tau$ if
\[
	|H(\tau)_{u,u}|^2 + |H(\tau)_{u,v}|^2 = 1,
\]
where $\alpha = H(\tau)_{u,u}$ and $\beta = H(\tau)_{u,v}$ denote the complex
amplitudes at $u$ and $v$, respectively \cite{Monterde2023Fractional}. The case $\beta \neq 0$ is called
\emph{proper} FR, while the case $\beta = 0$ is called \emph{periodicity}
relative to $u$. In particular, when $\alpha = 0$ (and hence $\beta = 1$), we say that the graph $G$
exhibits \emph{perfect state transfer} (PST) from $u$ to $v$. Note that, FR is
a generalization of both PST ($\alpha = 0$) and periodicity ($\beta = 0$).
Furthermore, since proper $(\alpha, \beta)$-FR occurs from $u$ to $v$ if and
only if proper $(-\frac{\bar{\alpha}\beta}{\bar{\beta}}, \beta)$-FR occurs from
$v$ to $u$ \cite[Proposition 4.1]{Chan2019Quantum}, we shall often say proper FR
occurs \emph{between} $u$ and $v$. We also say the graph $G$ itself is periodic
if there is time $\tau$ such that $|H(\tau)_{u, u}| = 1$ for all vertices $u$
\cite{Godsil2012State}.

FR was introduced in \cite{Rohith2015Visualizing}
as a physical mechanism in quantum state transfer.
Unlike PST, where the quantum state is transferred entirely from one vertex to another,
FR describes a phenomenon in which only a portion of the state is transferred,
while the remaining amplitude may persist or reappear at other vertices.
In quantum physics, fractional revival arises naturally in systems such as the infinite potential well,
where the time-evolved wave function can be expressed as a superposition of translated copies
of the initial state arranged in a parity-preserving manner
\cite{Aronstein1997Fractional}.
Dooley and Spiller \cite{Dooley2014Fractional} further investigated FR in the context of qubit interactions,
relating it to multiple Schr\"odinger-cat states and quantum carpets.
In recent years, FR has developed into a topic at the interface of mathematics and physics,
particularly within spectral graph theory and algebraic combinatorics.

The mathematical theory of quantum state transfer, encompassing both PST and
FR, has developed rapidly. PST was first introduced by Christandl et al.
\cite{Christandl2004Perfect} and has since been extensively studied. Godsil \cite{Godsil2012State} provided a comprehensive survey of PST in graphs and its associated spectral conditions. Several works on PST have focused on graphs arising from groups, see for examples \cite{Cao2020Perfect, Cao2021Perfect}. PST has also been investigated in graphs derived from rings.
Integral circulant graphs were characterized spectrally in \cite{Basic2010Perfect},
while unitary Cayley graphs admit PST only in the cases $K_2$ and $C_4$
\cite{Basic2009Perfect}.
Eigenvalue methods were extended to unitary Cayley graphs over finite local
rings \cite{Meemark2014Perfect},
and later generalized to finite commutative rings and gcd-graphs
\cite{Thongsomnuk2019Perfect},
with additional refinements in \cite{Cheung2011Perfect,Kendon2011Perfect}.

Subsequent research has focused on FR.
Fractional revival in XX quantum spin chains was studied in \cite{Genest2016Quantum},
where models based on isospectral deformations and Para-Krawtchouk polynomials
were shown to admit both PST and FR.
Balanced fractional revival and its connection with quantum walks on hypercubes
were examined in \cite{Bernard2018Graph}.
Symmetry of FR between vertex pairs was established in \cite{Chan2019Quantum},
while conditions for FR in paths and cycles were analyzed in
\cite{Chan2020Fractional}.
A general framework for FR in spin networks, extending the notion of cospectral vertices,
was developed in \cite{Chan2022Fundamentals}, and it was shown that FR may occur between non-cospectral vertices in certain infinite graphs \cite{Godsil2022Fractional}.
Characterizations of FR in weighted graphs, including twin vertices and double cones,
were obtained in \cite{Monterde2023Fractional}.
Further studies investigated FR in semi-Cayley and Cayley graphs over finite
abelian groups \cite{Wang2025Fractional,Wang2024Fractional},
and in unitary Cayley graphs over $\mathbb Z_n$ \cite{Soni2025Quantum,Jitngam2026Quantum}.

Consider the ring $\mathbb{Z}_n$ of integers modulo $n$.
Recall that a nonzero element $a \in \mathbb{Z}_n$ is called a zero divisor
if there exists a nonzero element $b \in \mathbb{Z}_n$ such that $ab = 0$.
The \emph{zero-divisor graph} $\Gamma(\mathbb{Z}_n)$ is a simple undirected graph
whose vertex set consists of all nonzero zero divisors of $\mathbb{Z}_n$,
where two distinct vertices $u$ and $v$ are adjacent if and only if
$uv = 0$.

Zero-divisor graphs have been extensively studied and play an important role
in algebraic graph theory. Motivated by the preceding discussion,
we investigate PST and FR on the zero-divisor graph $\Gamma(\mathbb{Z}_n)$ in this paper.
In Section~\ref{zdg}, we present the structural properties of $\Gamma(\mathbb{Z}_n)$ and survey the graph adjacency spectrum.
A sufficient condition for $\Gamma(\mathbb{Z}_n)$ to admit PST
is established in Section~\ref{PST}.
In Section~\ref{FR}, we derive necessary and sufficient conditions for two
vertices of $\Gamma(\mathbb{Z}_n)$ to be strongly copsectral. Furthermore, we
study FR on bipartite $\Gamma(\mathbb{Z}_n)$ and prove the non-existence of
FR on $\Gamma(\mathbb{Z}_{p^2q})$.

\section{Preliminaries}\label{Pre}
Throughout this paper, unless otherwise stated, we assume that $G$ is a finite,
connected, simple, and unweighted graph with vertex set $V(G)$ and edge set $E(G)$.
We write $u \sim v$ to indicate that vertices $u$ and $v$ are adjacent, and we
denote the neighborhood of a vertex $u \in V(G)$ by $N_G(u)$. The \emph{adjacency
	matrix} $A(G)$ of $G$ has entries $A(G)_{u, v} = 1$ if $u \sim v$, and $0$ otherwise.
Because $G$ is undirected, $A(G)$ is real and symmetric; its \emph{characteristic
	polynomial} is defined as $\phi(A(G), x) = \det(xI - A(G))$, and its \emph{spectrum},
denoted $\sigma(G)$, is the multiset of its eigenvalues. For distinct eigenvalues
$\lambda_1, \dots, \lambda_k$ with multiplicities $m_1, \dots, m_k$, we write
$\sigma(G) = \{ \lambda_1^{m_1}, \dots, \lambda_k^{m_k} \}$. The matrices $J$, $I$,
and $O$ denote the all-ones, identity, and zero matrices of appropriate order,
respectively, while the characteristic vector $\mathbf{e}_u$ of $u$ has a $1$ in
the $u$-th entry and $0$ elsewhere. If $H$ is another graph, we write $G \cong H$
to indicate that $G$ and $H$ are isomorphic.

We denote the \emph{complete graph} and \emph{path graph} on $n$ vertices by
$K_n$ and $P_n$, respectively. The \emph{complete bipartite graph} with partite
sets of size $n_1$ and $n_2$ is denoted as $K_{n_1, n_2}$, with the \emph{star
	graph} $S_n \cong K_{1,{n-1}}$ as a special case. The \emph{complement}
$\overline{G}$ of $G$ has edge set $E(\overline{G}) = \{uv : u \not\sim v
	\text{ in } G\}$. The complement of $K_n$ is the \emph{null graph},
$\overline{K}_n$. A graph is \emph{regular} if all vertices have the same
degree.

The \emph{Cartesian product} $G \square H$ has vertex set $V(G) \times V(H)$,
where $(u, v) \sim (u', v')$ in $G\square H$ if either $u = u'$ and $v \sim v'$
in $H$, or $v = v'$ and $u \sim u'$ in $G$. The eigenvalues of $G \square H$ are
the sums $\lambda + \mu$, where $\lambda \in \sigma(G)$ and $\mu \in \sigma(H)$,
taking multiplicities into account.

A partition $\pi = \{V_1, V_2, \dots, V_k\}$ of $V(G)$ is \emph{equitable} if
the number of neighbors in $V_j$ of any vertex $u \in V_i$ is a constant
$q_{ij}$, independent of $u$. (Note that context will distinguish this use of
$\pi$ from the mathematical constant $\pi$.) The \emph{quotient graph} $G/\pi$
is defined as a weighted directed graph whose vertices are the cells of $\pi$,
with adjacency matrix given by $A(G/\pi)_{i, j} = q_{ij}$. The \emph{discrete
	partition} (where each cell is a singleton) is always equitable, while the
\emph{trivial partition} (consisting of a single cell) is equitable if and only
if $G$ is regular. Because the join of any two equitable partitions is equitable,
any arbitrary partition of $V(G)$ refines the unique \emph{coarsest equitable
	partition} (CEP) \cite[Section 8]{Godsil2012State}.

In a connected graph $G$, let $d(u,v)$ be the shortest path distance between
vertices $u$ and $v$, and let $\epsilon(u) = \max_{v} d(u,v)$ be the
\emph{eccentricity} of $u$. The \emph{distance partition} of $G$ with respect
to $u$ is $\Delta_u = \{D_0, \dots, D_{\epsilon(u)}\}$, where $D_i = \{v
	\in V(G) : d(u, v) = i\}$. Note that $\Delta_u$ need not be equitable.

For a partition $\pi = \{V_1, \dots, V_k\}$ of $V(G)$, its \emph{characteristic
	matrix} $S$ is the $|V(G)| \times k$ matrix whose $j$-th column is the
characteristic vector of the cell $V_j$; that is, $S_{u,j} = 1$ if $u \in V_j$
and $0$ otherwise. Normalizing these columns to unit length yields $\hat{S}$.
Consequently, $\hat{S}^T\hat{S} = I_{k}$ and $S\hat{S}^T$ is a block diagonal
matrix with $i$-th block $\frac{1}{|V_i|}J_{|V_i|}$. A vertex $u$ is a singleton
cell in $\pi$ if and only if $S\hat{S}^T\mathbf{e}_u = \mathbf{e}_u$.

\begin{lem}{\cite[Lemma 8.1]{Godsil2012State}}\label{equitable_char}
	For a partition $\pi$ with normalized characteristic matrix $\hat{S}$. The
	following statements are equivalent:
	\begin{enumerate}[label=(\alph*)]
		\item $\pi$ is an equitable partition.
		\item The column space of $\hat{S}$ is $A$-invariant.
		\item There exists a matrix $C$ of order $|\pi| \times |\pi|$ such that
		      $A\hat{S} = \hat{S}C$.
		\item $A$ and $S\hat{S}^T$ commute.
	\end{enumerate}
\end{lem}

When $\pi$ is equitable, left-multiplying the equation $A\hat{S} = \hat{S}C$
by $\hat{S}^T$ yields $C = \hat{S}^T A \hat{S}$. If we denote the standard
quotient matrix by $Q = A(G/\pi)$, then $AS = SQ$ and $C = DQD^{-1}$,
where $D = \text{diag}(\sqrt{|V_{1}|}, \dots, \sqrt{|V_{k}|})$. Thus, $C$
is a symmetrization of $Q$ via a similarity transformation, with entries
given by $C_{ij} = \sqrt{q_{ij}q_{ji}}$. We refer to the (weighted)
graph with adjacency matrix $C$ as the \emph{symmetrized quotient graph}
\cite{Fan2013Pretty}.

\section{The Zero-Divisor Graph $\Gamma(\mathbb{Z}_n)$}\label{zdg}

For an integer $n > 1$, let $Z(\mathbb{Z}_n)$ denote the set of all zero
divisors of $\mathbb{Z}_n$ and $Z^*(\mathbb{Z}_n) = Z(\mathbb{Z}_n) \setminus \{0\}$.
The zero-divisor graph $\Gamma(\mathbb{Z}_n)$ has the vertex set $V(\Gamma(\mathbb{Z}_n)) = Z^*(\mathbb{Z}_n)$, and recall that two distinct vertices $u$ and $v$ are
adjacent if and only if $uv = 0$ in $\mathbb{Z}_n$. Notably, the zero-divisor
graph of any commutative ring is connected with a diameter of at most
$3$ \cite{Anderson1999On}. This section details the structural properties and
adjacency spectrum of $\Gamma(\mathbb{Z}_n)$, which serve to characterize FR.

Following the notation in \cite{Bajaj2022On}, let $n$ be a positive integer with prime
factorization $n = p_1^{\alpha_1} p_2^{\alpha_2} \cdots p_k^{\alpha_k}$. The
total number of divisors of $n$ is $d(n) = \prod_{i=1}^{k} (\alpha_i + 1)$.
Consequently, the number of distinct proper divisors (excluding $1$ and $n$) is
$\xi = d(n) - 2$.

Because the vertices of $\Gamma(\mathbb{Z}_n)$ are precisely the nonzero zero
divisors of $\mathbb{Z}_n$, the order of the graph is:
\[
	|V(\Gamma(\mathbb{Z}_n))| = n - \varphi(n) - 1
	= \sum_{i=1}^{\xi} \varphi\left(\frac{n}{d_i}\right),
\]
where $d_1, d_2, \dots, d_{\xi}$ are the distinct proper divisors of $n$, and
$\varphi$ denotes the \emph{Euler totient function}. Since every nonzero element
of $\mathbb{Z}_{p}$ is a unit for any prime $p$, the vertex set
$V(\Gamma(\mathbb{Z}_p))$ is empty. Accordingly, we restrict our study to
composite integers $n \ge 4$ throughout this paper.

The architecture of $\Gamma(\mathbb{Z}_n)$ is determined by its divisor structure. We partition $V(\Gamma(\mathbb{Z}_n))$ into disjoint cells based on the greatest common divisor of each element with $n$. Specifically, for each proper divisor $d_i$ of $n$, we define:
\[
	V_{d_i} = \{ x \in \mathbb{Z}_n : \gcd(x, n) = d_i \}.
\]

\begin{figure}[ht]
	\centering
	\begin{subfigure}[b]{0.45\textwidth}
		\centering
		\includegraphics[width=\textwidth]{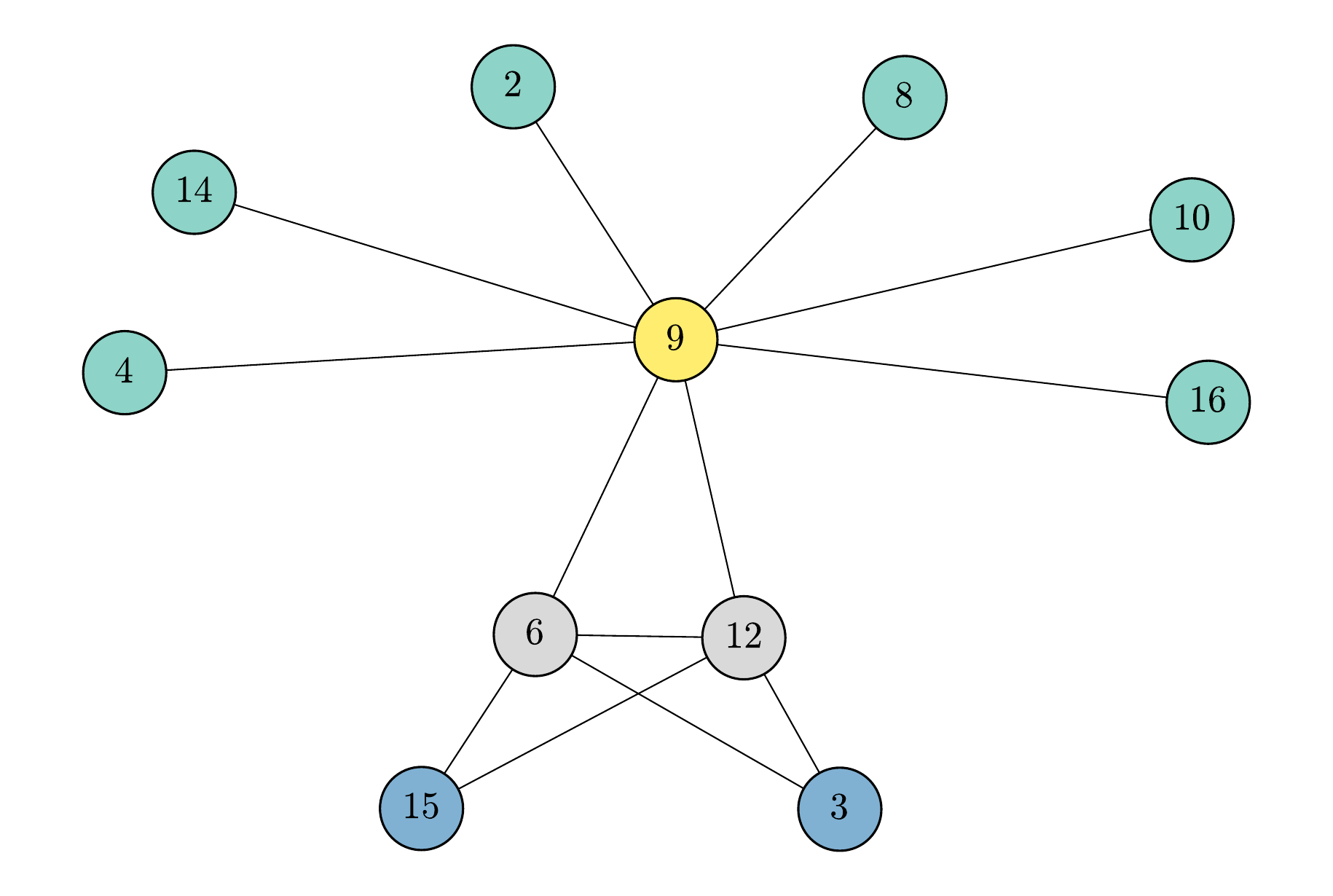}
		\caption{The zero-divisor graph $\Gamma(\mathbb{Z}_{18})$}
		\label{gamma_z18}
	\end{subfigure}
	\hfill
	\begin{subfigure}[b]{0.45\textwidth}
		\centering
		\includegraphics[width=\textwidth]{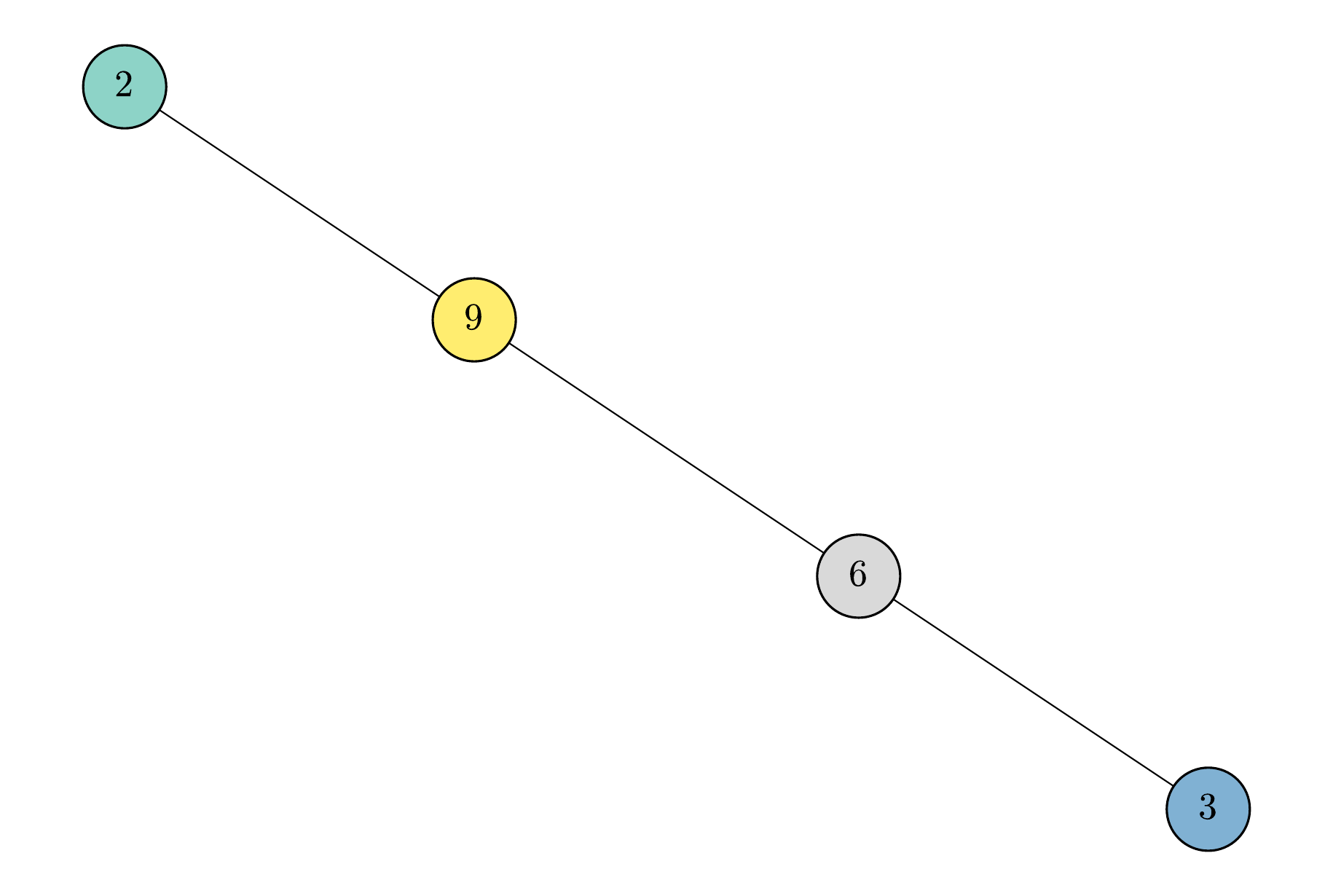}
		\caption{The simple graph $\Upsilon_{18}$}
		\label{upsilon_z18}
	\end{subfigure}
	\caption{The zero-divisor graph of $\mathbb{Z}_{18}$ and its associated divisor graph.}
	\label{combined_z18}
\end{figure}

\noindent Since $\gcd(x, n) = d_i$ if and only if $\gcd(x/d_i, n/d_i) = 1$, the
cardinality of each cell is $|V_{d_i}| = \varphi\left(n/d_i\right)$. For distinct
cells $V_{d_i}$ and $V_{d_j}$, a vertex $u \in V_{d_i}$ is adjacent to $v \in
	V_{d_j}$ if and only if $n \mid d_i d_j$. When this holds, every vertex in
$V_{d_i}$ is adjacent to every vertex in $V_{d_j}$. Furthermore, the
\textit{induced subgraph} $\Gamma(V_{d_i})$ is either a complete graph
$K_{|V_{d_i}|}$ if $n \mid d_i^2$, or a null graph $\overline{K}_{|V_{d_i}|}$
otherwise \cite[Corollary 2.5]{Chattopadhyay2020Laplacian} (see \cref{gamma_z18}).

Let $\mathcal{D} = \{d_1, d_2, \dots, d_{\xi}\}$ be the set of distinct proper
divisors of $n$, and let $\pi_{\mathcal{D}} = \{V_{d_1}, V_{d_2}, \dots,
	V_{d_{\xi}}\}$ be the corresponding vertex partition of $\Gamma(\mathbb{Z}_n)$.
We define a simple graph $\Upsilon_n$ on the vertex set $\mathcal{D}$, where
distinct vertices $d_i, d_j \in \mathcal{D}$ are adjacent if and only if $n \mid
	d_id_j$. This construction yields a decomposition of $\Gamma(\mathbb{Z}_n)$ via
the \emph{generalized join graph}. Recall that for a graph $G$ and a collection
of disjoint graphs $\{H_i\}$, the generalized join graph $G[H_1, \dots, H_k]$ is
formed by replacing each vertex $v_i \in V(G)$ with $H_i$, and joining all
vertices of $H_i$ to all vertices of $H_j$ whenever $v_i \sim v_j$ in $G$
\cite[Definition 2.1]{Chattopadhyay2020Laplacian}. Accordingly,
$\Gamma(\mathbb{Z}_n)$ satisfies the isomorphism:
\begin{equation}\label{zdg_iso}
	\Gamma(\mathbb{Z}_n) \cong \Upsilon_n[\Gamma(V_{d_1}), \Gamma(V_{d_2}),
		\dots, \Gamma(V_{d_{\xi}})].
\end{equation}

\noindent Using \cref{zdg_iso}, Bajaj et al. \cite{Bajaj2022On} derived the
adjacency spectrum of $\Gamma(\mathbb{Z}_n)$. Recall that each induced subgraph
$\Gamma(V_{d_t})$ is a regular graph with degree $r_t = \varphi(n/d_t) - 1$ if
it is a complete graph, and $r_t = 0$ if it is a null graph. Since these induced
subgraphs are strictly complete or null, the spectrum of $\Gamma(\mathbb{Z}_n)$
consists of $0$ and $-1$ (each possibly with multiplicity zero), togerther with
the eigenvalues of the $\xi \times \xi$ real symmetric matrix $C(\Upsilon_n)$,
whose entries $c_{ij}$ are defined as:
\begin{equation}\label{sym_quot_matrix}
	c_{ij} = \begin{cases}
		r_i, & \text{if } i = j,                                 \\[.5em]
		\sqrt{\varphi\left(\displaystyle\frac{n}{d_i}\right)\varphi\left(\displaystyle\frac{n}{d_j}\right)},
		     & \text{if } d_i \sim d_j \text{ in } \Upsilon_{n}, \\[.5em]
		0,   & \text{otherwise}.
	\end{cases}
\end{equation}

\begin{thm}{\cite[Theorem 3.3]{Bajaj2022On}} \label{spectrum1}
	Let $d_1, d_2, \dots, d_{\xi}$ be the distinct proper divisors of a positive
	integer $n$. Suppose that the vertex partition $\pi_{\mathcal{D}} =
		\{V_{d_i}\}_{i=1}^{\xi}$ is ordered such that the first $t$ cells (where
	$0 \le t \le \xi$) induce null graphs, and the remaining cells induce
	complete graphs in $\Gamma(\mathbb{Z}_n)$. Then the adjacency spectrum is
	given by:
	\[
		\sigma(\Gamma(\mathbb{Z}_n)) = \{0^{m_1}, -1^{m_2}\}
		\cup \sigma(C(\Upsilon_n)),
	\]
	where the multiplicities $m_1$ and $m_2$ are:
	\[
		m_1 = \sum_{i=1}^{t} \varphi\left(\frac{n}{d_i}\right) -
		t \quad \text{and} \quad m_2 = \sum_{i=t+1}^{\xi}
		\varphi\left(\frac{n}{d_i}\right) + t - \xi,
	\]
	with the convention that $m_1 = 0$ if $t = 0$.
\end{thm}

As the partition $\pi_{\mathcal{D}} = \{V_{d_i}\}_{i=1}^{\xi}$ is equitable, the
matrix $C(\Upsilon_n)$ is precisely the symmetrized quotient matrix
$C = \hat{S}^TA\hat{S}$ from \cref{equitable_char}. Moreover, for $n\neq 4$, all the eigenvalues
of $C(\Upsilon_n)$ are nonzero \cite[Theorem 3.5]{Bajaj2022On} (note that for
$n=4$, $C(\Upsilon_4)$ is the zero matrix, yielding an eigenvalue of $0$).

\begin{ex}
	Consider $\Gamma(\mathbb{Z}_{18})$ from \cref{combined_z18}. The proper
	divisors are $\mathcal{D} = \{2, 9, 6, 3\}$, giving $\xi = 4$. The cell
	cardinalities are $|V_2|=6$, $|V_9|=1$, $|V_6|=2$, and $|V_3|=2$. The cells $V_2, V_3$, and $V_9$ induce null graphs ($t=3$), while $V_6$ induces a complete graph.

	Applying \cref{spectrum1}, the multiplicities are $m_1 = 6$ and $m_2 = 1$.
	Evaluating \cref{sym_quot_matrix} yields the $4 \times 4$ symmetrized
	quotient matrix:
	$$
		C(\Upsilon_{18}) = \begin{pmatrix}
			0        & \sqrt{6} & 0        & 0 \\
			\sqrt{6} & 0        & \sqrt{2} & 0 \\
			0        & \sqrt{2} & 1        & 2 \\
			0        & 0        & 2        & 0
		\end{pmatrix}.
	$$
	Consequently, $\sigma(\Gamma(\mathbb{Z}_{18})) = \{0^{6}, -1^{1}\} \cup
		\sigma(C(\Upsilon_{18}))$.
\end{ex}

The structured eigenvalues of $\Gamma(\mathbb{Z}_n)$ come directly from the
presence of twin vertices within its cells. Two distinct vertices $u$ and $v$
in a graph $G$ are \emph{twins} if $N_G(u) \setminus \{v\} = N_G(v) \setminus
	\{u\}$. They are called \emph{true twins} if they are adjacent, and
\emph{false twins} otherwise. A subset  $T \subseteq V(G)$ with at least two
vertices is a \emph{set of twins} in $G$ if the vertices in $T$ are pairwise
twins.

Consider $G$ as a weighted graph where $\omega$ denotes the loop weight and
$\eta$ denotes the edge weight between any two vertices in $T$. By
\cite[Lemma 1]{Monterde2023Fractional}, if $u, v \in T$, then $\mathbf{e}_u -
	\mathbf{e}_v$ is an eigenvector of the adjacency matrix $A(G)$ with
corresponding eigenvalue $\theta = \omega - \eta$.

To apply this result to the cells of $\pi_{\mathcal{D}}$, we note that twin
vertices in $\Gamma(\mathbb{Z}_n)$ must necessarily belong to the same cell.
Indeed, if $u \in V_{d_1}$ and $v \in V_{d_2}$ for $d_1 \neq d_2$, the proof of
\cref{cell-vs-gcd} implies that $u$ and $v$ have distinct degrees and,
consequently, cannot be twins. Conversely, if $u$ and $v$ belong to the same
cell $V_d$ of $\pi_{\mathcal{D}}$, they share the same neighborhood, making them
twins. Thus, we have the following characterization:

\begin{lem}\label{cell-vs-twins}
Let $u$ and $v$ be distinct vertices of $\Gamma(\mathbb{Z}_n)$. Then $u$ and $v$
are twins if and only if they belong to the same cell of the equitable
partition $\pi_{\mathcal{D}}$.
\end{lem}

In our case, since  $\Gamma(\mathbb{Z}_n)$ is a simple graph, $\omega = 0$.
Furthermore, for any cell $V_{d_i}$ of size at least $2$, its vertices form a
set of twins where $\mathbf{e}_u - \mathbf{e}_v$ is an eigenvector. The
corresponding eigenvalue is $\theta = -1$ if $n \mid d_i^2$ (true twins,
$\eta =1$), and $\theta = 0$ otherwise (false twins, $\eta = 0$).

\section{Perfect state transfer}\label{PST}

We investigate perfect state transfer on $\Gamma(\mathbb{Z}_n)$ using the theory
of equitable partitions. By \cite[Corollary 2]{Monterde2023Fractional}, provided that
$|V(G)| \ge 3$, proper FR (and thus PST) cannot occur within twin sets of size $3$
or greater, nor between a twin and a non-twin. Because each cell $V_{d_i}$ of size
at least $2$ forms a set of twins, these general properties immediately rule out
PST between distinct cells, and within any cell where $|V_{d_i}| = \varphi(n/d_i) \ge 3$.

While these twin bounds effectively filter out non-candidate pairs, the theory
of equitable partitions provides a framework that independently recovers these
non-existence results and more importantly, establish conditions under which PST
occurs. We begin by showing that vertices grouped within the same cell of any
equitable partition of $\Gamma(\mathbb{Z}_n)$ must share the same greatest
common divisor with $n$.

\begin{lem}\label{cell-vs-gcd}
	Let $\Gamma(\mathbb{Z}_n)$ be the zero-divisor graph of $\mathbb{Z}_n$, and
	let $\pi$ be any equitable partition of its vertex set. If $u$ and $v$ are
	vertices belonging to the same cell of $\pi$, then $\gcd(u, n) = \gcd(v, n)$.
\end{lem}

\begin{proof}
	Let $u, v \in V(\Gamma(\mathbb{Z}_n))$ belong to the same cell of an equitable
	partition $\pi$. By definition, $u$ and $v$ must have the same degree. For
	any vertex $x \in V(\Gamma(\mathbb{Z}_n))$ with $d = \gcd(x, n)$, its degree
	is determined by the number of nonzero multiples of $\frac{n}{d}$ in
	$\mathbb{Z}_n$. There are exactly $d - 1$ such multiples. To avoid self-loops
	in the simple graph $\Gamma(\mathbb{Z}_n)$, we must subtract $1$ from this
	count if $x^2 \equiv 0 \pmod n$, a condition equivalent to $n \mid d^2$.
	Thus, the degree of $x$ is given by:
	\[
		\deg(x) = \begin{cases} d - 1 & \text{if } n \nmid d^2 \\ d - 2 &
              \text{if } n \mid d^2.\end{cases}
	\]

	Suppose, for the sake of contradiction, that $\gcd(u,n) \neq \gcd(v,n)$.
	Let $d_1 = \gcd(u, n)$ and $d_2 = \gcd(v, n)$, and assume without loss of
	generality that $d_1 < d_2$. From the piecewise function above, the maximum
	possible degree of $u$ is $d_1 - 1$, and the minimum possible degree of
	$v$ is $d_2 - 2$. For $\deg(u) = \deg(v)$ to hold, it is necessary that:
	\[
		d_2 - 2 \le d_1 - 1 \implies d_2 \le d_1 + 1.
	\]

	Since $d_1 < d_2$, we must have $d_2 = d_1 + 1$. This strict equality forces
	$\deg(u) = d_1 - 1$ and $\deg(v) = d_2 - 2$, which in turn requires $n \nmid
		d_1^2$ and $n \mid d_2^2$. Note that consecutive integers are coprime, so
	$\gcd(d_1, d_2) = 1$. Since $d_1 \mid n$ and $d_2 \mid n$, their coprimality
	ensures that $d_1 d_2 \mid n$. Substituting $d_2 = d_1 + 1$ yields:
	\[
		d_1(d_1 + 1) \mid n.
	\]

	Furthermore, since $n \mid d_2^2$, it follows that $n \mid (d_1 + 1)^2$. By
	the transitivity of divisibility:
	\[
		d_1(d_1 + 1) \mid (d_1 + 1)^2.
	\]

	Canceling the nonzero factor $(d_1 + 1)$ implies $d_1 \mid (d_1 + 1)$, and
	consequently, $d_1 \mid 1$. This forces $d_1 = 1$, contradicting the
	requirement that any zero divisor $u$ in $\Gamma(\mathbb{Z}_n)$ must satisfy
	$\gcd(u, n) > 1$. This is a contradiction, and we conclude that $\gcd(u, n)
		= \gcd(v, n)$.
\end{proof}

In the proof of \cite[Corollary 9.3]{Godsil2012State}, Godsil showed that if
$\pi_u$ and $\pi_v$ are the CEPs refining the initial partitions $\{\{u\}, V(G)
	\setminus \{u\}\}$ and $\{\{v\}, V(G) \setminus \{v\}\}$ respectively, then the
existence of PST between $u$ and $v$ implies that $\pi_u = \pi_v$. This follows
by observing that the transition matrix $H(t)$ commutes with the orthogonal
projection onto the subspace of functions constant on the cells of $\pi_u$.
Consequently, PST from $u$ to $v$ forces $v$ to be a singleton cell in $\pi_u$,
and by symmetry, the two partitions must coincide. Utilizing this fact, our next
theorems establish conditions under which a vertex pair $\{u, v\}$ in
$\Gamma(\mathbb{Z}_n)$ fails to admit PST.

\begin{thm}\label{no-pst-1}
	Suppose that $|V(\Gamma(\mathbb{Z}_n))| > 2$, and let $u$ and $v$ be two
	vertices of \,$\Gamma(\mathbb{Z}_n)$ such that $\gcd{(u, n)} \neq \gcd{(v, n)}$.
	Then perfect state transfer cannot occur between $u$ and $v$.
\end{thm}

\begin{proof}
	Let $d_s = \gcd{(u, n)}$ and $d_t = \gcd{(v, n)}$, and let $\pi_u$ and $\pi_v$
	denote the CEPs refining the partitions $\{\{u\}, V(G) \setminus \{u\}\}$
	and $\{\{v\}, V(G) \setminus \{v\}\}$, respectively. By \cref{cell-vs-gcd},
	any cell of an equitable partition of $\Gamma(\mathbb{Z}_n)$ must contain
	vertices that share the same gcd with $n$. Therefore, $\pi_u = \left\{\{u\},
		V_{d_s} \setminus \{u\}\right\} \displaystyle\bigcup_{i\neq s}V_{d_i}$ and
	$\pi_v = \left\{\{v\}, V_{d_t} \setminus \{v\}\right\}
		\displaystyle\bigcup_{i\neq t}V_{d_i}$. Since $\pi_u \neq \pi_v$, PST
	cannot occur between $u$ and $v$.
\end{proof}

\begin{thm}\label{no-pst-2}
	Let $d_i$ be a proper divisor of \,$n$, and consider vertices $u, v \in
		V_{d_i}$. If\, $|V_{d_i}| > 2$, then perfect state transfer cannot occur
	between $u$ and $v$.
\end{thm}

\begin{proof}
	By an argument similar to the proof of \cref{no-pst-1}, we obtain

	\begin{align*}
		\pi_u & = \left\{\{u\}, V_{d_i} \setminus \{u\}\right\}
		\displaystyle\bigcup_{j\neq i}V_{d_j}, \text{ and}      \\
		\pi_v & = \left\{\{v\}, V_{d_i} \setminus \{v\}\right\}
		\displaystyle\bigcup_{j\neq i}V_{d_j}.
	\end{align*}

	Observe that $\pi_u \neq \pi_v$; thus, the claim follows.
\end{proof}

\begin{rmk}
	The results in \cref{no-pst-1}  and \cref{no-pst-2} are consistent with the
	general properties of twin sets from \cite[Corollary 2]{Monterde2023Fractional}.
\end{rmk}

It can be deduced from \cref{no-pst-1} and \cref{no-pst-2} that PST cannot occur
between any vertex pair in $\Gamma(\mathbb{Z}_n)$ whenever $n$ is not a multiple
of $3$ or $4$ (e.g., $10, 50, 70, 125$). To investigate the remaining cases, we
state the following key theorem due to Bachman et al.~\cite{Bachman2012Perfect}.

\begin{thm}\cite[Theorem 1]{Bachman2012Perfect}\label{pst-quotient}
	Let $G$ be a graph with an equitable partition $\pi$. If vertices $u$ and
	$v$ belong to singleton cells in $\pi$, then for any time $t$,
	\[
		(e^{-itA(G)})_{u,v} = (e^{-itA(G/\pi)})_{u,v}.
	\]
	Consequently, $G$ admits perfect state transfer between $u$ and $v$ if and
	only if the quotient graph $G/\pi$ admits perfect state transfer between
	the corresponding vertices at the same time $t$.
\end{thm}

Using this result, the following theorem determines the values of $n$ for which
PST occurs and explicitly identifies the vertex pairs exhibiting this phenomenon.

\begin{thm}\label{has-pst}
	The zero-divisor graph $\Gamma(\mathbb{Z}_n)$ admits perfect state transfer
	whenever $n = 8, n =9$, or $n = 3p$, where $p \neq 3$ is a prime; in
	the latter case, the transfer occurs between vertices $p$ and $2p$ at the
	minimum time $\tau = \frac{\pi}{\sqrt{2(p-1)}}$.
\end{thm}

\begin{proof}
	For $n = 8$, the graph $\Gamma(\mathbb{Z}_8)$ is isomorphic to $P_3$, which
	is well known to admit PST between its end vertices, namely $2$ and $6$, at
	$\tau = \frac{\pi}{\sqrt{2}}$; similarly, for $n = 9$, $\Gamma(\mathbb{Z}_n)$
	is isomorphic to $P_2$, which admits PST between its  end vertices at $\tau =
		\frac{\pi}{2}$. Now, suppose that $n = 3p$ for $p\neq 3$ . Then,
    $\Gamma(\mathbb{Z}_{3p})$ is isomorphic to the complete bipartite
    graph $K_{2, p-1}$. Let $\pi = \left\{\{p\}, \{2p\},
		V(\Gamma(\mathbb{Z}_{3p}))\setminus \{p, 2p\}\right\}$
	be an equitable partition of $V(\Gamma(\mathbb{Z}_{3p}))$. Let $C$ be the
	adjacency matrix of the symmetrized quotient graph, defined by:
	\[
		C = \begin{pmatrix}
			0          & 0          & \sqrt{p-1} \\
			0          & 0          & \sqrt{p-1} \\
			\sqrt{p-1} & \sqrt{p-1} & 0
		\end{pmatrix}.
	\]
	Since $C = \sqrt{p-1}\,A(P_3)$, it follows that PST occurs between vertices
	$p$ and $2p$ at minimum time $\tau = \frac{\pi}{\sqrt{2(p-1)}}$ in the quotient
	graphs $\Gamma(\mathbb{Z}_{3p})/\pi$. Therefore, by \cref{pst-quotient}, the
	result follows.
\end{proof}

\begin{rmk}
	We note a subtle distinction regarding the necessary conditions for PST
	discussed in \cite[Corollary 9.3]{Godsil2012State, Godsil2017State}. While
	the proof of the corollary successfully establishes that PST between
	vertices $u$ and $v$ requires $\pi_u = \pi_v$, the corollary statement
	itself claims the stronger condition of $\Delta_u = \Delta_v$, which appears
	to be a typographical error. This stronger claim does not hold in general.
	For example, consider the Cartesian product $G = P_2 \square K_{1,4}$, where
	$V(P_2) = \{0, 1\}$ and $V(K_{1,4}) = \{c, l_1, l_2, l_3, l_4\}$ (see \cref{p2-k14}).

	\begin{figure}[htbp]
		\centering
		\includegraphics[width=.6\linewidth]{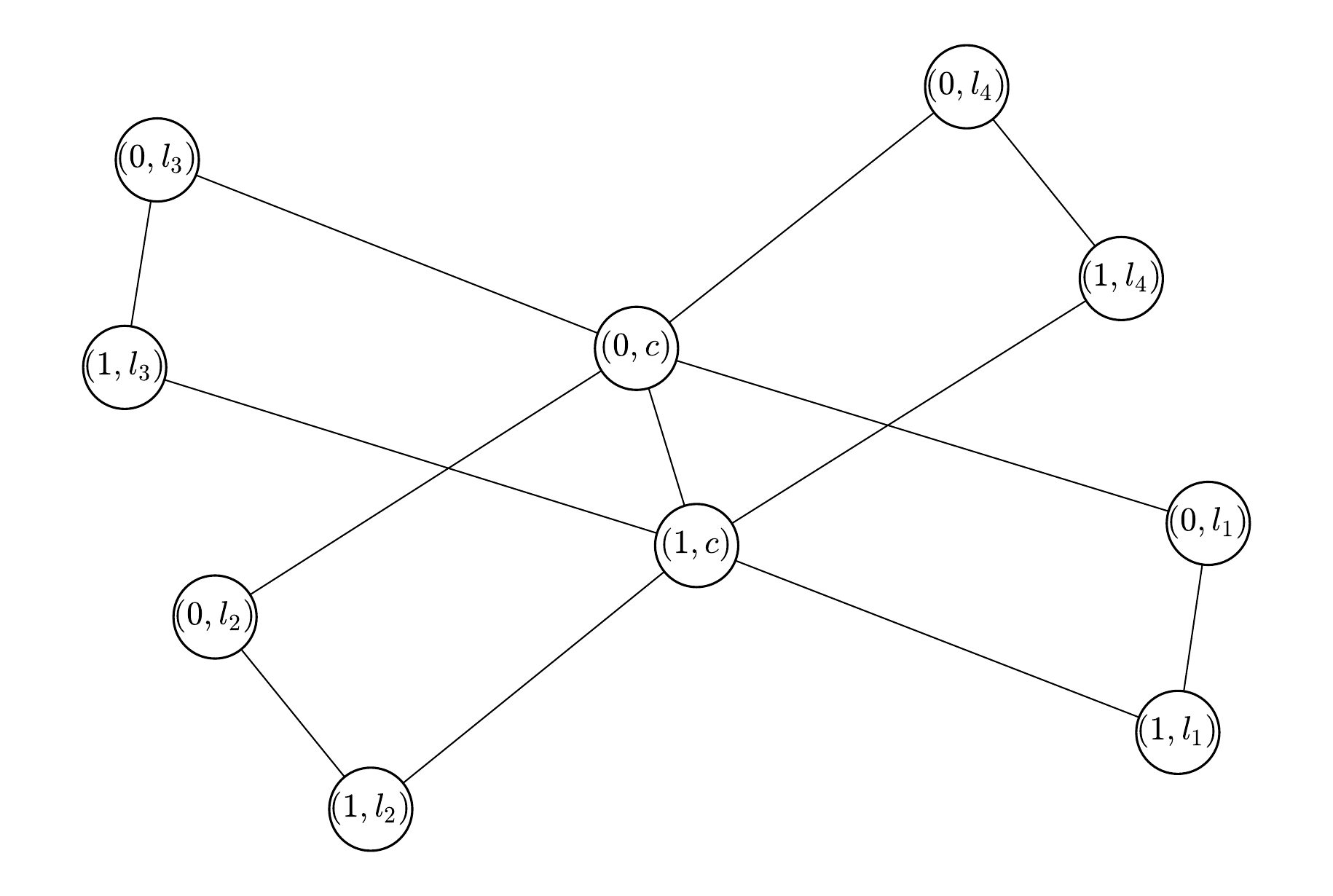}
		\caption{The Cartesian product graph $P_2 \square K_{1,4}$}
		\label{p2-k14}
	\end{figure}

	It is well known that $P_2$ has PST between $0$ and $1$ at time $t =
		\frac{\pi}{2}$ with a transition magnitude of $|-i| = 1$. Furthermore,
	since $\sigma(K_{1,4})$ consists of even integers, $K_{1, 4}$ is periodic at
	the central vertex $c$ at time $t = \frac{\pi}{2}$ with the transition
	amplitude $\left[ H_{A(K_{1,4})}\left(\frac{\pi}{2}\right) \right]_{c,c} = -1$.
	Using the property that the transition matrix of a Cartesian product is the
	Kronecker product of its factors \cite[Lemma 4.2]{Godsil2012State}, we have
	$H_{A(G)}(t) = H_{A(P_2)}(t) \otimes H_{A(K_{1,4})}(t)$, it follows that $G$
	admits PST between $(0, c)$ and $(1, c)$ at time $t = \frac{\pi}{2}$.

	However, the distance partitions for these vertices are not identical.
	Specifically:
	\begin{align*}
		\Delta_{(0,c)} & = \{\{(0, c)\}, \{(1,c)\} \cup L_0, L_1\}, \text{ and} \\
		\Delta_{(1,c)} & = \{\{(1, c)\}, \{(0,c)\} \cup L_1, L_0\},
	\end{align*}
	where
	\begin{align*}
		L_0 & = \{(0, l_1), (0, l_2), (0, l_3), (0, l_4)\}, \text{ and} \\
		L_1 & = \{(1, l_1), (1, l_2), (1, l_3), (1, l_4)\}.
	\end{align*}
	Since $\{(0,c)\} \in \Delta_{(0,c)}$ but $\{(0,c)\} \notin \Delta_{(1,c)}$,
	we have $\Delta_{(0,c)} \neq \Delta_{(1,c)}$.

	In contrast, the CEPS for these vertices coincide:
	\[
		\pi_{(0,c)} = \pi_{(1,c)} = \{\{(0,c)\}, \{(1,c)\}, L_0, L_1\}.
	\]

	This provides a counterexample to the claim that $\Delta_u = \Delta_v$ is a
	necessary condition for PST.
\end{rmk}

\section{Fractional revival}\label{FR}

As discussed in \cref{PST}, the only candidates for proper FR are the twin sets
of size $2$, where PST occurs if $n = 8$, $n=9$ or $n = 3p$. However, if $n
	\notin \{8, 9, 3p\}$, it remains unclear whether the vertices in the twin sets of
size $2$ exhibit the phenomenon. For instance, if $n = 18$, it is not
immediately apparent whether the pairs $\{3, 15\}$ and $\{6, 12\}$ exhibit PST.
Although the general case for $n$ is not fully determined, the structure of
such pairs allows for a partial characterization.

Let $A(G)$ have the spectral decomposition $\sum_{j}\lambda_jE_j$. The
\emph{eigenvalue support} of a vertex $u$ is defined as the set $\Phi_u = \{\lambda_j:
	E_j\mathbf{e}_u \neq \mathbf{0}\}$. Vertices $u$ and $v$ are \emph{cospectral}
if $(E_j)_{u, u} = (E_j)_{v, v}$ for each $j$, and \emph{parallel} if
$E_j\mathbf{e}_u = c_jE_j\mathbf{e}_v$ for some constants $c_j$. In the stronger
case where $E_{j}\mathbf{e}_u = \pm E_j\mathbf{e}_v$ for each $j$, the vertices
are said to be \emph{strongly cospectral}. Under this condition, the support
$\Phi_u = \Phi_v$ can be written as $\Phi_u = \Phi_{u, v}^{+} \cup \Phi_{u,
		v}^{-}$, where
\[
	\Phi_{u,v}^{+} = \{\lambda_j: E_j\mathbf{e}_u = E_j\mathbf{e}_v\}\;
	\text{and}\;\Phi_{u,v}^{-} = \{\lambda_j: E_j\mathbf{e}_u = -E_j\mathbf{e}_v\}
\]

By \cite[Lemma 4]{Monterde2023Fractional}, proper FR between twin vertices $u$
and $v$ requires them to be strongly cospectral. Since a twin pair is strongly
cospectral if and only if they are parallel \cite[Corollary 3.8]{Monterde2022Strong},
the occurrence of proper FR between twins requires parallelism. Therefore, our
task reduces to establishing the specific conditions under which these twin
pairs are parallel. To characterize parallel vertices in $\Gamma(\mathbb{Z}_n)$,
we first establish the following lemma.

\begin{lem}\label{equitable_parallel}
	Let $A$ be the adjacency matrix of $G$, and let $Q$ be the adjacency matrix
	of the quotient graph $G/\pi$ for some equitable partition $\pi$. If
	$\lambda$ is an eigenvalue of $A$ that is not an eigenvalue of $Q$, then any
	eigenvector $\mathbf{v}$ associated with $\lambda$ sums to zero on every
	cell of $\pi$.
\end{lem}

\begin{proof}
	Let $S$ be the characteristic matrix of $\pi$. Since $AS = SQ$, the column
	space of $S$, denoted by $\mathcal{C}(S)$, is $A$-invariant. Since $A$ is
	symmetric, its orthogonal complement $\mathcal{C}(S)^\perp$ is also
	$A$-invariant. We uniquely decompose the eigenvector $\mathbf{v}$ as
	$\mathbf{v} = \mathbf{x} + \mathbf{y}$, where $\mathbf{x} \in \mathcal{C}(S)$
	and $\mathbf{y} \in \mathcal{C}(S)^\perp$.

	The eigenvalue equation $A\mathbf{v} = \lambda\mathbf{v}$ yields $A\mathbf{x}
		+ A\mathbf{y} = \lambda\mathbf{x} + \lambda\mathbf{y}$. By the
	$A$-invariance of the subspaces and the uniqueness of the orthogonal
	decomposition, it follows that $A\mathbf{x} = \lambda\mathbf{x}$. Since
	$\mathbf{x} \in \mathcal{C}(S)$, we may write $\mathbf{x} = S\mathbf{w}$
	for some $\mathbf{w} \in \mathbb{R}^k$. Then,
	\begin{equation*}
		SQ\mathbf{w} = AS\mathbf{w} = A\mathbf{x} = \lambda\mathbf{x} =
		\lambda S\mathbf{w}.
	\end{equation*}
	Because $S$ has full column rank, $SQ\mathbf{w} = \lambda S\mathbf{w}$
	implies $Q\mathbf{w} = \lambda\mathbf{w}$. However, $\lambda$ is not an
	eigenvalue of $Q$ by hypothesis, so we must have $\mathbf{w} = \mathbf{0}$,
	which implies $\mathbf{x} = \mathbf{0}$. Thus, $\mathbf{v} = \mathbf{y} \in
		\mathcal{C}(S)^\perp$, completing the proof.
\end{proof}

Monterde’s work provides a general bound on the size of twin sets that can admit
parallel vertices. Specifically, \cite[Corollary 3.10]{Monterde2022Strong} shows
that parallelism cannot occur when $T \ge 3$. Our next results demonstrate that
for the graph $\Gamma(\mathbb{Z}_n)$, when $|T| = 2$ and the vertices are false
twins, they are parallel.

\begin{thm}\label{null-parallel}
	Let $V_d \in \pi_{\mathcal{D}}$ be a cell of size $|V_d| \ge 2$ such that
	the induced subgraph $\Gamma(V_d) \cong \overline{K}_{|V_d|}$. Then, any
	two distinct vertices in $V_d$ are parallel if and only if $|V_d| = 2$.
\end{thm}

\begin{proof}
	Let $u, v \in V_d$ be distinct false twins. First, we consider the case
	$\lambda \neq 0$. Let $\mathbf{x}$ be any eigenvector associated with
	$\lambda$. Because false twins share identical neighborhoods, their
	corresponding rows in the adjacency matrix $A = A(\Gamma(\mathbb{Z}_n))$
	are equal. Evaluating the $u$-th and $v$-th entries of $A\mathbf{x} =
		\lambda\mathbf{x}$ yields:

	\[\lambda (\mathbf{e}_u^T \mathbf{x}) = \mathbf{e}_u^T (A\mathbf{x}) =
		(\mathbf{e}_u^T A) \mathbf{x} = (\mathbf{e}_v^T A) \mathbf{x} =
		\mathbf{e}_v^T (A\mathbf{x}) = \lambda (\mathbf{e}_v^T \mathbf{x}).
	\]
	Since $\lambda \neq 0$, dividing by $\lambda$ yields $\mathbf{e}_u^T\mathbf{x}
		= \mathbf{e}_v^T\mathbf{x}$. Because this equality holds for every
	eigenvector in the $\lambda$-eigenspace $\mathcal{E}_\lambda$, the projections
	of $\mathbf{e}_u$ and $\mathbf{e}_v$ onto $\mathcal{E}_\lambda$ are identical.
	Thus, $E_\lambda \mathbf{e}_u = E_\lambda \mathbf{e}_v$.

	Consider the case $\lambda = 0$. Let $S$ be the characteristic matrix of the
	equitable partition $\pi_{\mathcal{D}}$. Following the argument in
	\cref{equitable_parallel}, we decompose $\mathbf{e}_u$ as follows:

	\[
		\mathbf{e}_u = \underbrace{\frac{1}{|V_d|}
			\mathbf{1}_{V_d}}_{\mathbf{m}} + \underbrace{\left( \mathbf{e}_u -
			\frac{1}{|V_d|} \mathbf{1}_{V_d} \right)}_{\mathbf{z}},
	\]
	where $\mathbf{1}_{V_d}$ denotes the characteristic vector of the cell $V_d$.

	Let $E_0$ denote the orthogonal projector onto the null space $\mathcal{N}(A)$.
	Applying $E_0$ to both sides of our decomposition yields:
	\[
		E_0 \mathbf{e}_u = E_0 \mathbf{m} + E_0 \mathbf{z}.
	\]

	Because all vertices in $V_d$ share identical columns in $A$, we know
	$A\mathbf{z} = \mathbf{0}$, meaning $\mathbf{z} \in \mathcal{N}(A)$ and
	$E_0 \mathbf{z} = \mathbf{z}$. Meanwhile, $\mathbf{m}$ is constant on the
	cells of $\pi_{\mathcal{D}}$ and therefore it lies in $\mathcal{C}(S)$.
	By \cref{equitable_parallel}, the absence of $0$ as an eigenvalue of the
	quotient matrix guarantees $\mathcal{N}(A) \perp \mathcal{C}(S)$. Thus,
	$E_0 \mathbf{m} = \mathbf{0}$. Applying $E_0$ to our decomposition yields
    \begin{equation}\label{charvec-projection}
		E_0 \mathbf{e}_u = \mathbf{e}_u - \frac{1}{|V_d|} \mathbf{1}_{V_d}
    \end{equation}

	For $u$ and $v$ to be parallel in the null space, we require
	$E_0 \mathbf{e}_u = c E_0 \mathbf{e}_v$ for some $c \neq 0$. Substituting
	our projection formula yields
	\[
		\mathbf{e}_u - \frac{1}{|V_d|} \mathbf{1}_{V_d} = c \left( \mathbf{e}_v -
		\frac{1}{|V_d|} \mathbf{1}_{V_d} \right).
	\]
	Evaluating this equation at the $u$-th coordinate gives $1 - \frac{1}{|V_d|}
		= c \left(0 - \frac{1}{|V_d|}\right)$, which simplifies to $c = 1 -
		|V_d|$. Conversely, evaluating at the $v$-th coordinate gives
	$0 - \frac{1}{|V_d|} = c \left(1 - \frac{1}{|V_d|}\right)$, which simplifies
	to $c = \frac{-1}{|V_d|-1}$.

	Equating these two required values for $c$ yields $(|V_d| - 1)^2 = 1$. Since
	$|V_d| \geq 2$, this forces $|V_d| = 2$ (yielding $c = -1$). Therefore,
	$u$ and $v$ are parallel in the null space, and consequently parallel, if
	and only if $|V_d| = 2$.
\end{proof}

\begin{prop}\label{complete-parallel}
	Let $V_d \in \pi_{\mathcal{D}}$ be a cell of size $|V_d| \ge 2$ such that
	the induced subgraph $\Gamma(V_d) \cong K_{|V_d|}$. If $-1$ is not an
	eigenvalue of the associated quotient matrix, then any two distinct vertices
	in $V_d$ are parallel if and only if $|V_d| = 2$.
\end{prop}

\begin{proof}
	Proceeding analogously, the adjacency rows of true twins satisfy
	$(\mathbf{e}_u - \mathbf{e}_v)^T A = -(\mathbf{e}_u - \mathbf{e}_v)^T$.
	Right-multiplying by any eigenvector $\mathbf{x} \in \mathcal{E}_\lambda$
	yields $\lambda(\mathbf{e}_u - \mathbf{e}_v)^T \mathbf{x} = -(\mathbf{e}_u
		- \mathbf{e}_v)^T \mathbf{x}$, which simplifies to
	$(\lambda + 1)(\mathbf{e}_u - \mathbf{e}_v)^T \mathbf{x} = 0$. Consequently,
	for all $\lambda \neq -1$, we have $\mathbf{e}_u^T \mathbf{x} =
		\mathbf{e}_v^T \mathbf{x}$. Since this holds for every vector in the
	eigenspace, it guarantees $E_\lambda \mathbf{e}_u = E_\lambda \mathbf{e}_v$.
	For $\lambda = -1 $, using the same decomposition $\mathbf{e}_u =
		\mathbf{m} + \mathbf{z}$, where $\mathbf{m} = \frac{1}{|V_d|}
		\mathbf{1}_{V_d}$ and $\mathbf{z} = \mathbf{e}_u - \mathbf{m}$.
	Let $E_{-1}$ be the orthogonal projector onto the eigenspace $\mathcal{E}_{-1}$.
	Because $V_d$ forms a complete graph, we have $A\mathbf{z} = -\mathbf{z}$,
	meaning $E_{-1}\mathbf{z} = \mathbf{z}$. Furthermore, since $-1$ is not an
	eigenvalue of the quotient matrix, \cref{equitable_parallel} ensures
	$\mathcal{E}_{-1}\perp \mathcal{C}(S)$, forcing $E_{-1}\mathbf{m} =
		\mathbf{0}$. Thus, $E_{-1}\mathbf{e}_u = \mathbf{z}$. By the exact same
	support argument used previously, distinct vertices in $V_d$ are parallel if
	and only if $|V_d| = 2$.
\end{proof}

\begin{thm}\label{iff-strongly-cospectral}
Suppose that $-1$ is not an eigenvalue of the associated quotient matrix. Two
vertices $u$ and $v$ in $\Gamma(\mathbb{Z}_n)$ is strongly cospectral if and
only if they form a cell $V_d = \{u, v\}$ of $\pi_{\mathcal{D}}$ that is a twin
set of size $2$.
\end{thm}

\begin{proof}
    Suppose that $u$ and $v$ form a cell $V_d$ of size $2$ of the equitable partition
    $\pi_{\mathcal{D}}$. If $V_d$ is a set of false twins, then $u$ and $v$ are
        parallel by \cref{null-parallel}. Analogously, provided that $-1$ is not
        an eigenvalue of the quotient matrix, if $V_d$ consists of true
        twins, then $u$ and $v$ are parallel by \cref{complete-parallel}.
        According to \cite[Corollary 3.8]{Monterde2022Strong}, $u$ and $v$ are
        strongly cospectral.

    Conversely, suppose that the vertices $u$ and $v$ in $\Gamma(\mathbb{Z}_n)$
    are strongly cospectral. We must show that they form a cell $V_d
    = \{u, v\}$ of $\pi_{\mathcal{D}}$ that is a twin set of size $2$.

    If $u$ and $v$ belong to the same cell $V_d \in \pi_{\mathcal{D}}$, then
    they are pairwise twins. In this case, the results of \cref{null-parallel}
    and \cref{complete-parallel} require that $|V_d| = 2$ for parallelism. Thus,
    the claim follows.

    Suppose, for the sake of contradiction, that $u$ and $v$ belong to distinct
    cells $V_{d_i}$ and $V_{d_j}$ of $\pi_{\mathcal{D}}$. By
    \cref{cell-vs-twins}, vertices in distinct cells of $\pi_{\mathcal{D}}$
    cannot be twins. By \cite[Theorem 3.9(2)]{Monterde2022Strong}, the
    projections of $u$ and $v$ onto the twin eigenspace $E_{\theta}$ (where
    $\theta \in \{0, -1\}$) are not proportional for any $c \in \mathbb{R}$.
    Consequently, $u$ and $v$ are not parallel, which contradicts the requirement
    for strong cospectrality. This shows that $u$ and $v$ must stay in the same
    cell $V_d$.
\end{proof}

We now show how the induced subgraph $\Gamma(V_d)$ restricts the spectral
support of the vertices in $V_d$, showing that certain eigenvalues are
necessarily excluded.

\begin{thm}\label{notin-complete-support}
	Let $V_d$ be a cell of $\pi_{\mathcal{D}}$ such that $\Gamma(V_d) \cong
    K_{|V_d|}$, and suppose that $0 \in \sigma(\Gamma(\mathbb{Z}_n))\setminus
    \sigma(C(\Upsilon_n))$. Then
	for any $u \in V_d$, $E_0\mathbf{e}_u = \mathbf{0}$; that is, $0 \notin
		\Phi_u$.
\end{thm}

\begin{proof}
	As in previous proofs, we decompose the characteristic vector as
	$\mathbf{e}_u = \mathbf{m} + \mathbf{z}$, where $\mathbf{m} \in \mathcal{C}(S)$
	and $\mathbf{z}$ is a zero-sum vector supported entirely on $V_d$. Let $E_0$
	denote the orthogonal projector onto $\mathcal{N}(A)$. Because $\Gamma(V_d)
		\cong K_{V_d}$, we have $A\mathbf{z} = -\mathbf{z}$, which implies
	$E_0\mathbf{z} = \mathbf{0}$. Furthermore, since $0$ is not an eigenvalue of
	the quotient matrix, we have $E_0\mathbf{m} = \mathbf{0}$. Therefore,
	$E_0\mathbf{e}_{u} = \mathbf{0}$.
\end{proof}

\begin{prop}\label{notin-null-support}
	Let $V_d$ be a cell of $\pi_{\mathcal{D}}$ such that $\Gamma(V_d) \cong
		\overline{K}_{|V_d|}$. Suppose that  $-1 \in \sigma(\Gamma(\mathbb{Z}_n))
		\setminus \sigma(C(\Upsilon_n))$. Then for any $u
		\in V_d$, $E_{-1}\mathbf{e}_u = \mathbf{0}$; that is, $-1 \notin \Phi_u$.
\end{prop}

\begin{proof}
	The proof proceeds analogously. Using the same decomposition $\mathbf{e}_u =
		\mathbf{m} + \mathbf{z}$, the condition $\Gamma(V_d) \cong
		\overline{K}_{|V_d|}$ gives $A\mathbf{z} = \mathbf{0}$, yielding
	$E_{-1}\mathbf{z} = \mathbf{0}$. Furthermore, since $-1$ is not an
	eigenvalue of the quotient matrix, we have $E_{-1}\mathbf{m} = \mathbf{0}$.
	Therefore, $E_{-1}\mathbf{e}_u = \mathbf{0}$.
\end{proof}

\begin{rmk}\label{no-0-and-1}
Suppose that $|V(\Gamma(\mathbb{Z}_n))| > 1$ and let $V_d = \{u\}$ be a
singleton cell of $\pi_{\mathcal{D}}$. Provided that $-1 \notin
\sigma(\Gamma(\mathbb{Z}_n))$, the support of $u$ contains neither $0$ nor $-1$,
if they exist in the spectrum of $\Gamma(\mathbb{Z}_n)$. Indeed, this
follows from \cref{charvec-projection} with $|V_{d}| = 1$, together with 
\cref{notin-null-support} and \cref{notin-complete-support}.
\end{rmk}

Suppose that $V(\Gamma(\mathbb{Z}_n)) \ge 3$. Since $\Gamma(\mathbb{Z}_n)$ is connected
and the proof of \cref{null-parallel} shows that $\mathbf{e}_u - \mathbf{e}_v$
is an eigenvector for the eigenvalue $\theta = 0$ corresponding to the strongly
cospectral false twins $u$ and $v$, it follows from
\cite[Theorem 3.4]{Monterde2022Strong} that $|\Phi_u| \ge 3$ and
$|\Phi_{u, v}^{-}| = 1$. Consequently, the set $\Phi_{u, v}^+$ must contain at
least $2$ distinct eigenvalues, yielding the decomposition
\[
	\Phi_u = \Phi_{u, v}^{-} \cup \Phi_{u, v}^{+}
	= \{0\} \cup \{\lambda_1, \dots, \lambda_r\},
\]
where $r \ge 2$ and $\lambda_1 > \lambda_2$.

Analogously, for true twins (where $\theta = -1$), provided that $-1 \notin
	\sigma(C(\Upsilon_n))$, \cref{complete-parallel} and
\cite[Theorem 3.4]{Monterde2022Strong} establish the similarly decomposition
\[
	\Phi_u = \Phi_{u, v}^{-} \cup \Phi_{u, v}^{+}
	= \{-1\} \cup \{\lambda_1, \dots, \lambda_r\}.
\]

Proper FR is characterized through the eigenvalue support of twin vertices in
the following results due to Monterde.

\begin{thm}\cite[Theorem 6]{Monterde2023Fractional}\label{proper-fr-symmetry}
	Let $\phi(A(G), x) \in \mathbb{Z}[x]$, and suppose that $u$ and $v$ are twins in $G$.
	Proper fractional revival occurs between $u$ and $v$ if and only if they are
	strongly cospectral and one of the following conditions holds:
	\begin{enumerate}
		\item All elements in $\Phi_{uv}^+$ have the form $\frac{1}{2}(2\theta
			      + b_j\sqrt{\Delta})$, where $b_j$ is even and either
		      $\Delta = 1$ or $\Delta > 1$ is square-free, and $g \nmid
			      \frac{\lambda_1 - \theta}{\sqrt{\Delta}},$
		      where $g = \gcd\left( \frac{\lambda_1 - \lambda_2}{\sqrt{\Delta}},
			      \dots, \frac{\lambda_1 - \lambda_r}{\sqrt{\Delta}} \right)$.

		\item All elements in $\Phi_{uv}^+$ have the form
		      $\frac{1}{2}(a + b_j\sqrt{\Delta})$, where $a \neq 2\theta$,
		      $b_j$ is even, and $\Delta > 1$ is square-free.
	\end{enumerate}
\end{thm}

If $(\alpha, \beta)$-FR occurs between twin vertices $u$ and $v$ at time $\tau$,
there exist $\zeta, \gamma \in \mathbb{R}$ such that $\alpha =
	e^{i\zeta}\cos\gamma$ and $\beta = i e^{i\zeta}\sin\gamma$ \cite[Proposition
	5.1]{Chan2019Quantum}. The following corollary determines the specific values
of $\tau$ and $\gamma$ required for proper FR to occur.

\begin{cor}\cite[Corollary 4]{Monterde2023Fractional}\label{proper-fr-time}
	Suppose $u$ and $v$ are twins in $G$ that admit proper
	$(e^{i\zeta}\cos\gamma, e^{i\zeta}\sin\gamma)$-fractional revival at time
	$\tau$, then $\Phi_{u, v}^{+}$ satisfies the ratio condition. Specifically,
	if $p_j$ and $q_j$ are coprime integers such that
	$\frac{\lambda_1 - \lambda_j}{\lambda_1 - \lambda_2} = \frac{p_j}{q_j}$,
	then there exists an integer $k$ such that:
	\[
		\tau = \frac{2\pi qk}{\lambda_1 - \lambda_2}\quad\text{and}\quad
		\gamma\equiv qk\left(\frac{\lambda_1 - \theta}{\lambda_1
			-\lambda_2}\right)\pi \pmod \pi,
	\]
	where $q = \text{lcm}(q_2, \dots, q_n)$ and
	$q\left(\frac{\lambda_1 - \theta}{\lambda_1 - \lambda_2}\right)$ is not an
	integer.
\end{cor}

\begin{ex}
	Consider the graph $\Gamma(\mathbb{Z}_{21})$. By \cref{has-pst}, we know PST
	occurs between vertices $7$ and $14$ at time some time $\tau$. Since $\Phi_{7,
			14}^{+} = \{\lambda_1, \lambda_2\} = \{\sqrt{12}, -\sqrt{12}\}$ and
	$\Phi_{7, 14}^{-} = \{\theta\} = \{0\}$, it follows that $q = 1$. With
	$q = 1$ and $k = 1$, we have $\tau = \frac{\pi}{\sqrt{12}}$ and
	$\gamma = \frac{\pi}{2} \pmod \pi$.
\end{ex}

The following results by Monterde provide a characterization of periodicity.
For the sake of completeness, we restate the portion regarding the periodicity of
twin vertices below.

\begin{thm}\cite[Theorem 8]{Monterde2023Fractional}\label{periodic-symmetry}
	Let $\phi(A(G), x) \in \mathbb{Z}[x]$ and suppose $u$ and $v$ are twins in
	$G$ that admit fractional revival. Vertices $u$ and $v$ are periodic if
	and only if all elements in $\Phi_{u, v}^{+}$ have the form
	$\frac{1}{2}(2\theta + b_j\sqrt{\Delta})$, where $b_j$ is even
	and either $\Delta = 1$ or $\Delta > 1$ is square-free.
\end{thm}

\begin{cor}\cite[Corollary 8]{Monterde2023Fractional}
	Let $\phi(A(G),x) \in \mathbb{Z}[x]$ and suppose $u$ and $v$ are twins in $G$
	that admit fractional revival. If at least one element in
	$\Phi_{u, v}^{+}$ is an integer, then $u$ and $v$ are periodic.
\end{cor}

Since $\phi(A(\Gamma(\mathbb{Z}_n)), x) \in \mathbb{Z}[x]$, setting $\theta = 0$
for false twins and $\theta = -1$ for true twins in the \cref{periodic-symmetry}
yields the required symmetry conditions for their non-integral support
whenever FR occurs. Specifically, for any-non integer $\lambda$, the support
of a false twin $u$ is symmetric about $0$ in the sense that
$\lambda \in \Phi_{u} \implies -\lambda \in \Phi_{u}$. Furthermore, assuming
$-1 \notin \sigma(C(\Upsilon_n))$, the support of a
true twin $v$ is symmetric about $-1$, meaning that $\lambda
	\in \Phi_v \implies -2 -\lambda \in \Phi_v$.

\begin{ex}
	By the Perron-Frobenius Theorem, the Perron root $\rho \approx \text{3.399}$ is
	in the eigenvalue support of every vertex in $\Gamma(\mathbb{Z}_{18})$. However,
	the spectrum lacks the symmetric counterpart $-\rho \approx -3.399$ and $-2 -
		\rho \approx -5.399$. Consequently, the symmetry conditions fail for both
	$V_3 = \{3, 15\}$ (where $0 \in \Phi_{3, 15}^{-}$) and $V_6 = \{6, 12\}$ (where $-1 \in
		\Phi_{6, 12}^{-}$), which immediately rules out periodicity. Consequently,
	PST cannot occur between any two vertices of $\Gamma(\mathbb{Z}_{18})$.
\end{ex}

The following corollary is an immediate consequence of \Cref{iff-strongly-cospectral}:

\begin{cor}\label{no-proper-fr-pk}
Let $n = p^k$ for a prime $p$ and an integer $k\ge 2$. If $p > 3$, the graph
$\Gamma(\mathbb{Z}_{p^k})$ does not admit proper fractional revival between
any two vertices.
\end{cor}

\begin{proof}
    The set of distinct proper divisors of $n = p^k$ is given by $\mathcal{D} =
    \{p^i\mid 1 \le i \le k - 1\}$. It follows that for $p > 3$, $\varphi(n/d) >
    2$ for all $d \in \mathcal{D}$. If an eigenvalue $\theta \in \{0, -1\}$ is
    present in the spectrum of $\Gamma(\mathbb{Z}_{p^k})$, it will not be in the
    spectrum of the quotient matrix $C(\Upsilon_{p^k})$ (see \cref{no-minus-one-prime-odd}
    and \cref{no-minus-one-prime-even}). By \cref{iff-strongly-cospectral} and
    \cref{proper-fr-symmetry}, the claim follows.
\end{proof}

\begin{ex}\label{p-squared}
    Let $n = p^2$, where $p \ge 2$ is a prime. Note that the spectrum is given
    by $\sigma(\Gamma(\mathbb{Z}_{p^2})) = \{p - 2, -1^{(p-2)}\}$. Since the
    spectrum consists entirely of integers, it follows that $\Gamma(\mathbb{Z}_{p^2})$
    is periodic with minimum period $\tau = \frac{2\pi}{p - 1}$. For $p\neq 3$,
    by \cref{no-proper-fr-pk}, $\Gamma(\mathbb{Z}_{p^2})$ does not admit proper
    FR. Conversely, for $p = 3$, PST occurs between vertices in
    $\Gamma(\mathbb{Z}_9)$.
\end{ex}

\begin{ex}\label{p-cubed}
    Consider the case $n = p^3$. According to \cref{no-proper-fr-pk}, the graph
    $\Gamma(\mathbb{Z}_{p^3})$ has no proper FR for $p > 3$. If $n = 2^3$, the
    graph $\Gamma(\mathbb{Z}_8)$ admits PST (and thus proper FR) by \cref{has-pst}.
    Finally, for $n = 3^3$,  consider the pair $\{9, 18\}$, which is strongly cospectral.
    Given that the support is $\Phi_{9, 18}^{+} = \{\lambda_1, \lambda_2\} = \{4,
    -3\}$ and the shift is $\theta = -1$, it follows from \cref{proper-fr-time}
    with $q = 1$ and $k = 1$ that proper FR occurs at minimum time
    $\tau = \frac{2\pi}{7}$. Moreover, using $\gamma \equiv \frac{5\pi}{7} \pmod
    \pi$ obtained from \cref{proper-fr-time}, we find the transition amplitudes
    to be $\alpha = e^{i\zeta} \cos\left(\frac{5\pi}{7}\right) \quad \text{and}
    \quad \beta = e^{i\zeta} \sin\left(\frac{5\pi}{7}\right)$. Consequently, at
    time $\tau = \frac{2\pi}{7}$, the state is distributed between the two vertices with
    probabilities $|\alpha|^2 \approx 0.3887$ and $|\beta|^2 \approx 0.6113$.
\end{ex}

We now characterize FR on $\Gamma(\mathbb{Z}_n)$ for the bipartite case. According
to \cite[Theorem 2.3]{Bajaj2022On}, $\Gamma(\mathbb{Z}_n)$ is bipartite if and
only if $n \in \{8, 9, pq, 4q\}$, where $p$ and $q$ are distinct primes. We first
address the case $n = pq$.

\begin{thm}\label{fr-pq}
    Let $p, q$ be distinct primes. Then the graph $\Gamma(\mathbb{Z}_{pq})$ is
    periodic. Moreover, $\Gamma(\mathbb{Z}_{pq})$ admits proper fractional
    revival if and only if $p = 3$ or $q = 3$.
\end{thm}

\begin{proof}
    We first establish the periodicity of $\Gamma(\mathbb{Z}_{pq})$. Note that
    $\Gamma(\mathbb{Z}_{pq}) \cong K_{p-1, q-1}$, with spectrum
    $\left\{ \pm\sqrt{(p-1)(q-1)}, 0^{(p+q-4)}\right\}$. Since these
    eigenvalues are either integers or quadratic integers, it follows that
    $\Gamma(\mathbb{Z}_{pq})$ is periodic with minimum period $\tau =
    \frac{2\pi}{\sqrt{(p-1)(q-1)}}$ \cite[Theorem 5.2]{Godsil2012State}. We note,
    however, that the minimum period may differ across the vertices. If $p = 2$
    or $q = 2$, we may assume $p=2$ by symmetry. Then $\Gamma(\mathbb{Z}_{2q})
    \cong K_{1,q-1}$, where the central vertex $q$ forms a singleton cell of
    $\pi_{\mathcal{D}}$. Since the eigenvalue $0$ is not in the support of this
    central vertex (\cref{no-0-and-1}), its minimum period is $\tau = \frac{\pi}{\sqrt{q-1}}$.
    In contrast, the leaf vertices in $\Gamma(V_{2})$ require the alignment of
    the zero eigenvalue. Consequently, the minimum period at the leaf vertices
    is $\tau = \frac{2\pi}{\sqrt{q-1}}$. For $p, q > 2$, the induced subgraphs
    $\Gamma(V_{p})$ and $\Gamma(V_q)$ of $\Gamma(\mathbb{Z}_{pq})$ both are null.
    By \cref{periodic-symmetry} with $\theta = 0$, any vertex in $\Gamma(\mathbb{Z}_{pq})$
    is periodic with minimum period $\tau = \frac{2\pi}{\sqrt{(p-1)(q-1)}}$.

    The existence of proper FR between two vertices $u, v$ in
    $\Gamma(\mathbb{Z}_{pq})$ requires that they be strongly cospectral. If $p,q
    \neq 3$, the graph contains no equitable partition cell of size $2$, and thus
    no pair of vertices satisfies the requirement for strong cospectrality. Therefore,
    proper FR cannot occur by \cref{proper-fr-symmetry}. Conversely, if $p = 3$
    or $q = 3$, we may assume without loss of generality that $p = 3$. Then PST
    occurs between vertices $q$ and $2q$ by \cref{has-pst}, and thus the claim
    follows.
\end{proof}

\begin{thm}\label{fr-bipartite}
Suppose $\Gamma(\mathbb{Z}_n)$ is bipartite. Then $\Gamma(\mathbb{Z}_n)$ is
periodic if and only if $n\neq 4q$. Moreover, $\Gamma(\mathbb{Z}_n)$ exhibits
proper fractional revival if and only if $n \in \{8, 9, 3q\}$.
\end{thm}

\begin{proof}
    Assume that $\Gamma(\mathbb{Z}_n)$ is bipartite, we know that $n \in \{8,
    9, pq, 4q\}$ \cite[Theorem 2.3]{Bajaj2022On}. If $n \in \{8, 9, pq\}$, the
    spectrum of $\Gamma(\mathbb{Z}_n)$ consists of
    integers or quadratic integers; thus, $\Gamma(\mathbb{Z}_n)$ is periodic by
    \cite[Theorem 5.2]{Godsil2012State}. Conversely, if $n = 4q$,
    \cref{no-fr-p2q} establishes that $\Gamma(\mathbb{Z}_{4q})$ is not periodic. 

    For the second claim, the existence of proper FR for $n \in \{8, 9, 3q\}$
    follows from \cref{has-pst} and \cref{fr-pq}. Conversely, if $n = 4q$, the
    graph cannot exhibit proper FR by \cref{no-fr-p2q}. The case $n = pq$ where
    $p, q \neq 3$ follows from \cref{fr-pq}, and this completes our proof.
\end{proof}

Next, we consider the case $n = p^2q$, where the existence of FR is fully
characterized. To simplify notation, we let $\phi_C(x)$ denote $\phi(C(\Upsilon_n), x)$
in what follows. We begin with the following lemma.

\begin{lem}\label{no-quadratic-factor-p2q}
	Let $p$ and $q$ be distinct primes. Then the characteristic polynomial
	$\phi(C(\Upsilon_{p^2q}), x)$ cannot be factored into two quadratic factors.
\end{lem}

\begin{proof}
	By \cite[Example 3.1]{Bajaj2022On}, the characteristic polynomial
	of $C(\Upsilon_{p^2q})$ is given by
	\begin{equation}\label{charpoly-p2q}
		\begin{split}
			\phi_C(x) & = x^4 - (p-2)x^3 - 2p(p-1)(q-1)x^2 + \\&\quad\;
			p(p-1)(p-2)(q-1)x + p(p-1)^3(q-1)^2.
		\end{split}
	\end{equation}

	Reducing $\phi_C(x)$ modulo $p$ yields $\phi_C(x) \equiv x^4 + 2x^3
		\pmod p$. If $\phi_C(x)$ were to factor into two quadratics,
	their reductions modulo $p$ would necessarily be $x^2$ and $x^2 + 2x$.
	Because both lack a constant term modulo $p$, the constant term of $\phi_C(x)$
	would be a multiple of $p^2$. That is, we must have $p^2 \mid
		p(p-1)^3(q-1)^2$, which implies $p \mid (p-1)^3(q-1)^2$. Therefore, if
	$p \nmid (q-1)$, $\phi_C(x)$ cannot be factor into two quadratic
	polynomials. Thus, it suffices to consider only the case $p \mid (q-1)$.

	Suppose $\phi_C(x)$ factors into two monic quadratics over $\mathbb{Z}$,
	given by
	\begin{equation} \label{charpoly-factor-2}
		\phi_C(x) = (x^2 + ax + b)(x^2 + cx + d).
	\end{equation}
	Let $M = (p-1)(q-1), N = p-2, K = pM$. We assume $N > 0$. Note that if
	$N = 0$ (i.e., $p = 2)$, we have $\sigma(C(\Upsilon_{4q})) =
		\left\{\pm\sqrt{(q-1)(2\pm \sqrt{2})}\right\}$. Because these eigenvalues
	generate a field extension of degree $4$ over $\mathbb{Q}$, their
	characteristic polynomial is clearly irreducible over $\mathbb{Z}$.

	By expanding \cref{charpoly-factor-2} and equating the coefficients with
	those of \cref{charpoly-p2q}, we obtain
	the following system:
	\begin{subequations} \label{system}
		\begin{align}
			a + c      & = -N \label{sys_x3}                 \\
			ac + b + d & = -2K \label{sys_x2}                \\
			ad + bc    & = KN \label{sys_x1}                 \\
			bd         & = \frac{K^2(p-1)}{p} \label{sys_x0}
		\end{align}
	\end{subequations}

	By substituting $c = -(a+N)$ from \cref{sys_x3} into \cref{sys_x2},
	we obtain $b + d = a(a + N) - 2K$. Similarly, substituting \cref{sys_x3}
	into \cref{sys_x1} yields $b - d = -\frac{N(b + K)}{a}$. Note that $a \neq 0$;
	indeed, $a = 0$ implies $c = -N$ and $b = d = -K$, which reduces \cref{sys_x0}
    to the contradiction $K^2 = K^2(p-1)/p$ for $K > 0$. Furthermore, since $N =
    p-2$ is odd for all primes $p > 2$, $2a + N$ is nonzero for any $a \in
    \mathbb{Z}$. Combining the expressions for $b+d$ and $b-d$, we obtain:
	\begin{align}
		b & = \frac{a^2(a+N)}{2a+N} - K, \label{b_final} \\[1ex]
		d & = \frac{a(a+N)^2}{2a+N} - K. \label{d_final}
	\end{align}

	Defining $y = a(a+N)$ and noting the identity $(2a + N)^2 = N^2 + 4y$, we
    substitue \cref{b_final,d_final} into \cref{sys_x0} (with $M = K/p$) to
    obtain the following equation in $y$:
	\begin{equation}\label{equation-y}
		y^3 = pM(y - M)(N^2 + 4y).
	\end{equation}
	Since the right-hand side of \cref{equation-y} is congruent to $0 \pmod p$,
    if $y$ is an integer solution, then $y$ must be a multiple of $p$.

    Given that $p \mid (q-1)$, we may write $M = pA$ where $A$ is a positive
    multiple of $p-1$. Substituting this and $y = pz$ into \cref{equation-y}
    yields the following equation in $z \in \mathbb{Z}$:
	\begin{equation}\label{equation-z}
		z^3 = A(z - A)(N^2 + 4pz).
	\end{equation}

    To show that no integer solutions for $z$ exist, we consider three cases:

	\begin{enumerate}[label=\textbf{Case \arabic*:}, wide, labelindent=0pt]
        \item $z = 0$. Substituting $z = 0$ into \cref{equation-z} yields
            $0 = -A^2N^2$, which is a contradiction.

        \item $z < 0$. Let $z' = -z > 0$ and write $z' = kd,
            A = ld$ where $d = \gcd(z', A)$ and $\gcd(k, l) = 1$. Note that
            $d, k, l$ are all positive integers. Under these substitutions,
            \cref{equation-z} becomes:
            \begin{equation}\label{equation-k-neg-1}
                (kd)^3  = d^2l(l + k)(N^2 - 4pz')
            \end{equation}
            which simplifies to:
            \begin{equation}\label{equation-k-neg-2}
                dk^3  = l(l + k)(N^2 - 4pz')
            \end{equation}
            Since $\gcd{(k^3, l)} = 1$ and $\gcd{(k^3, l + k)} = 1$, it follows
            that $l(l + k)$ must divide $d$. Setting $d = nl(l+k)$ for some
            $n\in \mathbb{Z}^{+}$, we can write $z' = nlk(l + k)$ and $A =
            nl^2(l + k)$. Substituting the expressions for $d$ and $z'$ into
            \cref{equation-k-neg-2}, we have:
            \begin{align}\label{equation-k-neg-last}
                nl(l+k)k^3         & = l(l + k)(N^2 - 4pz') \nonumber  \\
                nk^3                 & = N^2 - 4pz' \nonumber            \\
                nk^3                 & =  N^2 - 4p(nlk(l + k)) \nonumber \\
                nk^3 + 4pnlk(l + k) & = (p-2)^2
            \end{align}
            Expanding the right-hand side of \cref{equation-k-neg-last} and
            isolating terms in $p$, we set $Z' = 4nlk(l + k)$ to obtain:
            \begin{equation}\label{equation-p}
                nk^3 - 4 = p(p - (Z' + 4)) 
            \end{equation}
            Assuming $nk^3 \le 4$. The \cref{equation-p} implies that either
            $p$ divides the nonzero difference $nk^3 - 4$, or $nk^3 = 4$. In
            the former case, we must have $p \le |nk^3 - 4| \le 3$; however,
            $p = 2$ is excluded by hypothesis, and $p = 3$ (possible only if
            $nk^3 = 1)$ requires $Z' = 0$, contradicting $Z' \ge 8$. If
            $nk^3 = 4$, then $(n,k)= (4,1)$ forces $p = Z' + 4$, which is a
            multiple of $4$ and thus not prime. Therefore, $nk^3 - 4 > 0$.

            \noindent Since $nk^3 - 4 > 0$ implies $p > Z' + 4$, we let $p =
            Z' + 4 + u$ for some integer $u\ge 1$. Substituting this expression
            into \cref{equation-p} and recalling that $Z' = 4nlk(l + k)$, we
            obtain:
            \begin{align}\label{equation-u}
                nk^3 -4 &= (4nlk(l + k) + 4 + u)u \nonumber \\
                nk(k^2 - 4ul(l + k)) &= (u + 2)^2.
            \end{align}
            Since the right-hand side of \cref{equation-u} is strictly
            positive, the integer $k^2 - 4ul(l + k)$ must be at least $1$.
            Rearranging this inequality yields
            \begin{equation}
                k(k - 4ul) > 4ul^2 > 0,
            \end{equation}
            which forces $k > 4ul$.

            \noindent Let $k = 4ul + w$, where the integer $w \ge 1$. Substituting
            this into \cref{equation-u} yields: 
            \begin{equation}\label{equation-u-2}
                n(4ul + w)(w^2 + 4ul(w - l)) = (u + 2)^2
            \end{equation}
            To conclude the case for $z < 0$, we show that all
            possible relationship between $w$ and $l$ leads to a contradiction.

            \noindent \textbf{Subcase 2.1: $w > l.$} Rearranging
            \cref{equation-u-2} yields a quadratic in $u$:
            \begin{equation}\label{equation-u-3}
                u^2[16nl^2(w-l) - 1] + u[4nlw(2w - l) - 4] + nw^3 - 4 = 0
            \end{equation}
            Since $w - l \ge 1$ and $l \ge 1$, every coefficient of
            \cref{equation-u-3} is strictly positive. Consequently, the
            equation can have no positive root $u \ge 1$.

            \noindent \textbf{Subcase 2.2: $w = l.$} Setting $L = nl^3$,
            \cref{equation-u-3} reduces to:
            \begin{equation}
                u^2 - 4(L - 1)u - (L - 4) = 0 \nonumber
            \end{equation}
            For $u$ to be an integer, its discriminant $\Delta' = L(4L -7)$ must
            be a perfect square. Letting $L(4L - 7) = Z^2$ and completing the
            square, we obtain the difference of squares:
            \begin{equation}
                (8L - 7)^2 - (4Z)^2 = 49 \nonumber
            \end{equation}
            The only integer factors of $49$ giving a valid solution are
            $(1, 49)$, which implies $L = 4$. Consequently, $nl^3 = 4$, which
            leads to $p = 39216$. This contradicts the primality of $p$.

            \noindent \textbf{Subcase 2.3:} $w < l.$ For the positivity of 
            \cref{equation-u-2} to hold, we must have:
            \begin{equation}
                w^2 + 4ul(w - l) = w^2 - 4ul(l - w) \ge 1.
            \end{equation}
            This inequality implies $4ul(l - w) < w^2$. Since $l - w \ge 1$, it
            follows that $4ul < w^2$, and consequently $2\sqrt{ul} < w$. Given
            our assumption that $w < l$, we obtain the bound:
            \begin{equation}
                2\sqrt{ul} < l \implies 4ul < l^2 \implies 4u < l.
            \end{equation}
            Recalling \cref{equation-u-2}, the positivity of the factors on the
            left-hand side implies $(u+2)^2 > 4ul + w > 4ul$. Substituting $l > 4u$
            then yields $(u+2)^2 > 16u^2$, or equivalently $15u^2 - 4u - 4 < 0$.
            This inequality has no integer solution for $u \ge 1$. Therefore, the
            case $w < l$ is impossible.

        \item $z > 0.$ As before, let $d = \gcd{(z, A)}$ and write $z = kd$ and
            $A = ld$, where $\gcd{(k, l)} = 1$. Since the right-hand side of
            \cref{equation-z} is positive, it follows that $z > A$, and thus $k
            > l \ge 1$. Proceeding analogously to the case $z < 0$, the same
            divisibility argument yields $d = nl(k - l)$ for some $n \in
            \mathbb{Z}^{+}$. Substituting these expressions into \cref{equation-z},
            we obtain:
            \begin{align}\label{equation-k-pos-1}
                nk(k^2 - 4pl(k - l)) & = (p-2)^2
            \end{align}
            Let $\Delta = k - l$. Substituting $k = l + \Delta$ into
            \cref{equation-k-pos-1} yields:
            \begin{align}\label{equation-l}
                (l + \Delta)^2 - 4pl\Delta &= \frac{(p-2)^2}{nk} \nonumber \\
                l^2 - (4p - 2)l\Delta + \Delta^2 &= \frac{(p-2)^2}{nk}
            \end{align}
            Let $f(l, \Delta) = l^2 - (4p - 2)l\Delta + \Delta^2$ denote the
            left-hand side of \cref{equation-l}. Since the right-hand side
            is positive, we must have $f(l, \Delta) > 0$. Because $f(l, \Delta)$
            is symmetric in its arguments, we may assume without loss of generality
            that $l \ge \Delta$. Dividing the condition $f > 0$ by $\Delta^2$
            and setting $r = \frac{l}{\Delta} \ge 1$, it follows that $r$ must
            lie outside the interval $[r_2, r_1]$ formed by the roots of the
            equation:
            \begin{equation*}
            r^2 - (4p - 2)r + 1 = 0
            \end{equation*}
            where $r_1$ and $r_2$ denote the larger and smaller roots, respectively.

            \noindent Because $r_1 > 4p - 3 > 5$ and $r \ge 1 > r_2$, we have $r > 4p -
            3$. Therefore, $k = (r + 1)\Delta > (4p - 2)\Delta$.
            Since $\Delta \ge 1$, we have $k \ge 4p - 1$. Substituting the
            bound $k \ge 4p - 1$ into the right-hand side of
            \cref{equation-l} and noting that $n \ge 1$, we obtain:
            \begin{equation}\label{ineq-contradiction}
                \frac{(p - 2)^2}{nk} \le \frac{(p - 2)^2}{4p - 1} =
                \frac{p^2 - 4p + 4}{4p - 1} < \frac{p}{4}. 
            \end{equation}
            To analyze the left-hand side, recall that $l > (4p - 3)\Delta$. We
            may write $l = (4p - 2)\Delta + u$ for some integer $u > -\Delta$.
            Substituting this into $f(l, \Delta)$ yields :
            \begin{equation}
            f(l, \Delta) = u^2 + (4p-2)u\Delta  + \Delta^2
            \end{equation}
            We now consider the possible values for the integer $u$.

            \noindent \textbf{Subcase 3.1: $u \ge 1.$} For $u \ge 1$, the
            left-hand side satisfies the bound:
            \begin{align*}
            f(l, \Delta) &= u^2 + (4p - 2)u\Delta + \Delta^2 \\
                    &\ge 1^2 + (4p - 2)\Delta + \Delta^2 \\
            &\ge 4p\Delta + (\Delta - 1)^2 \\
            &\ge 4p,
            \end{align*}
            where the last inequality holds for all $\Delta \ge 1$. Since $4p > p
            / 4$ for any $p > 3$, this contradicts the upper bound established
            in \cref{ineq-contradiction}.

            \noindent \textbf{Subcase 3.2: $u = 0$}. If $u = 0$, then $l =
            (4p - 2)\Delta$ and $f(l, \Delta) = \Delta^2$. It follows that $k =
            (4p - 1)\Delta$, and \cref{equation-l} requires
            $n(4p-1)\Delta^3 = (p-2)^2$. This implies that $(4p - 1)$ must
            divide $(p - 2)^2$. Using the identity $16(p-2)^2 =
            (4p-1)(4p-15) + 49$, we must have $(4p - 1) \mid 49$. However, for
            $p > 2$, the only divisor of $49$ such that $4p - 1 >
            7$ is $49$. Solving $4p - 1 = 49$ yields $p = 12.5$, which is not an
            integer.

            \noindent \textbf{Subcase 3.3: $u \le -1$.} Let $u' = -u$. Since
            $u > -\Delta$, we have the integer bounds $1 \le u' \le \Delta - 1$.
            Substituting this into $f(l, \Delta)$ yields:
            \begin{equation*}
            f(l, \Delta) =  (u')^2 - (4p-2)u'\Delta + \Delta^2.
            \end{equation*}

            \noindent We can view $f(l, \Delta)$ as a quadratic in $u'$. The vertex
            occurs at $u' = (2p - 1)\Delta $. For $p \ge 3$,
            this vertex lies to the right of the interval $[1, \Delta - 1]$.
            Consequently, the function is strictly decreasing on this domain, and
            its maximum value is attained at the endpoint $u' = 1$:
            \begin{equation*}
            f(l, \Delta) \le 1 - (4p - 2)\Delta + \Delta^2 .         
            \end{equation*}

            \noindent Since \cref{equation-l} requires $f(l, \Delta) \ge 1$, it
            follows that $\Delta^2 - (4p - 2)\Delta \ge 0$. Given $\Delta \ge 1$,
            dividing by $\Delta$ yields the lower bound $\Delta \ge 4p - 2$.
            Recalling that $k > (4p - 2)\Delta$ and substituting $\Delta \ge 4p - 2$,
            we get:
            \begin{equation*}
            k > (4p - 2)^2 = 16p^2 - 16p + 4.
            \end{equation*}

            \noindent From the relation $nkf(l, \Delta) = (p-2)^2$, it follows
            that  $nk \le (p-2)^2$, which forces the upper bound $k \le p^2 - 4p
            + 4$. Combining these inequalities yields:
            \begin{equation*}
            16p^2 - 16p + 4 < p^2 - 4p + 4 \implies 15p^2 - 12p < 0.
            \end{equation*}
            This is impossible for any $p \ge 3$. \qedhere
    \end{enumerate}
\end{proof}

\begin{thm}\label{no-fr-p2q}
	Let $p$ and $q$ be distinct primes. Then the graph $\Gamma(\mathbb{Z}_{p^2q})$
	does not admit fractional revival between any pair of distinct vertices.
\end{thm}

\begin{proof}
	Consider a set of twins $V_d = \{u, v\}$ in $\Gamma(\mathbb{Z}_{p^2q})$. Since
    neither $0$ nor $-1$ is an eigenvalue of $C(\Upsilon_{p^2q})$
    (see \cref{no-minus-one-p2q}), $u$ and $v$ are strongly cospectral by
    \cref{iff-strongly-cospectral}.
    Because $|V(\Gamma(\mathbb{Z}_{p^2q})| \ge 3$, the eigenvalue support
    satisfies $\Phi_{u, v}^{+} \ge 2$ by \cite[Theorem 3.4]{Monterde2022Strong}.
    By \cref{no-quadratic-factor-p2q}, the eigenvalues in $\Phi_{u, v}^{+}$
    are either roots of an irreducible cubic or quartic, or they consist of an
    integer and a root of an irreducible cubic. In either case, the support fails
    to satisfy both \cref{periodic-symmetry} and \cref{proper-fr-symmetry}. Thus,
    the claim follows.
\end{proof}

Some of our results thus far require that $-1 \notin
	\sigma(\Gamma(\mathbb{Z}_n)/\pi_{\mathcal{D}})$. While we conjecture this
holds broadly, a general proof remains challenging due to the fast growth of
the quotient matrix as the graph size $n$ increases. We next show that $-1$
is not in the spectrum of the symmetrized quotient matrix $C(\Upsilon_n)$
for $n \in \{pq, p^2q, p_1p_2p_3, p^k\}$. We note that the case $n = pq$ is
straightforward, as $\Gamma(\mathbb{Z}_{pq})$ is a complete bipartite graph
$K_{\varphi(p), \varphi(q)}$ with a known spectrum that not having $-1$ for
distinct primes $p, q$. 

\begin{thm}\label{no-minus-one-p2q}
	Let $n = p^2q$, where $p$ and $q$ are distinct primes. Then $-1 \notin
		\sigma(C(\Upsilon_n))$.
\end{thm}

\begin{proof}
	Let $\phi_C(x)$ be the characteristic polynomial of the
	quotient matrix, as derived in \cite[Example 3.1]{Bajaj2022On}. Evaluating
	$\phi_C(-1)$, we find $\phi_C(-1) = 1 + (p-2) - 2p(p-1)(q-1)
		- p(p-1)(p-2)(q-1) + p(p-1)^3(q-1)^2.$

	Factoring out the common term $(p-1)$, we obtain: \[
		\phi_C(-1) = (p-1) \left[ 1 - p^2(q-1) + p(p-1)^2(q-1)^2 \right].
	\]
	Since $p$ is prime, $(p-1) \ge 1$. Let $f(p, q) = p(p-1)^2(q-1)^2 -
		p^2(q-1)+1$. It suffices to show $f(p,q) \neq 0$ for all primes $p, q$.
	We write $f(p, q)$ as:
	\[
		f(p, q) = p(q-1) \left[ (p-1)^2(q-1) - p \right] + 1.
	\]

	Given $q \geq 2$, we note $(p-1)^2(q-1) - p \geq p^2 - 3p + 1$. For $p
		\geq 3$, this quadratic is strictly positive, implying $f(p, q) > 1$; for
	$p=2$, the expression becomes $2(q-1)(q-3) + 1$, which is nonzero for all
	prime $q$. It follows that $\phi_C(-1)$ never vanishes, and thus
	$-1 \notin \sigma(C(\Upsilon_n))$.
\end{proof}

\begin{thm}\label{no-minus-one-p1p2p3}
	Let $n = p_1p_2p_3$, where $p_1, p_2$, and $p_3$ are distinct primes. Then
	$-1\notin \sigma(C(\Upsilon_n))$.
\end{thm}

\begin{proof}
	Let $Q = \varphi(p_1p_2p_3)$ and $S = \varphi(p_1p_2) + \varphi(p_2p_3) +
		\varphi(p_3p_1)$. From \cite[Theorem 3.7]{Bajaj2022On}, $\phi_C(-1)$ is
	given by:
	\[
		(Q - 1)^3 - S(Q - 1) - 2Q = 0,
	\]
	which can be written as:
	\[
		(Q - 1)[(Q - 1)^2 - S - 2] = 2.
	\]
	Since $p_1, p_2, p_3$ are distinct primes, the smallest value for $Q$ is
	$8$. For any other set of distinct primes, $Q > 8$. Consequently, $Q - 1 \ge
		7$, which cannot divide $2$. Thus, $x = -1$ is not a root.
\end{proof}

\begin{thm}\label{no-minus-one-prime-odd}
	Let $n = p^k$ where $p$ is a prime and $k$ is odd. Then $-1 \notin
		\sigma(C(\Upsilon_n))$.
\end{thm}

\begin{proof}
	Let $k = 2m + 1$ for some $m \in \mathbb{Z}^+$, and let $\phi_C(x)$
	denote the characteristic polynomial of $C(\Upsilon_n)$. Following
	\cite[Theorem 4.1]{Bajaj2022On}, the polynomial is given by:
	\[
		\phi_C(x) = (-x)^m \prod_{t=1}^m \varphi(p^t) \det(X),
	\]
	where $X$ is the $m \times m$ tridiagonal matrix as defined in the theorem.

	Evaluating the expression at $x = -1$, the pre-factors $(-x)^m
		\prod \varphi(p^t)$ are strictly positive for $p \geq 2$. Thus, $-1$ is a
	root if and only if $\det(X) = 0$ at $x = -1$.

	Under the substitution $x = -1$, all terms in the definitions of the
	off-diagonal entries $b_j$ containing $(1+x)$ vanish, thus $X$ is
	a diagonal matrix. The diagonal entries are $a_1 = 1 - \varphi(p^{m+1})$ and
	$a_i = -\varphi(p^{m+i})$ for $i = 2, \dots, m$. Since $\varphi(p^{m+1}) \geq 2$
	for all primes $p$ and $m \geq 1$, it follows that $a_i < 0$ for all $i$.
	Consequently, $\det(X) = \prod_{i=1}^m a_i \neq 0$, which confirms that $-1$
	is not an eigenvalue.
\end{proof}

\begin{thm}\label{no-minus-one-prime-even}
	Let $n = p^k$ where $p$ is a prime and $k$ is even. Then $-1 \notin
		\sigma(C(\Upsilon_n))$.
\end{thm}

\begin{proof}
	For $n = p^{2m}$, the characteristic polynomial is defined as: \[
		\phi_C(x) = (-x)^{m-1} \prod_{t=1}^m \varphi(p^t) \det(X),
	\]
	where $X$ is the $m \times m$ tridiagonal matrix as defined in
	\cite[Theorem 4.2]{Bajaj2022On}.

	Substituting $x = -1$, the off-diagonal entries $b_j$ vanish for
	$j \geq 2$, while $b_1 = 1$ remains constant. This results in a block
	diagonal structure for $X$, comprising a $2 \times 2$ leading principal
	submatrix $M$ and a diagonal trailing part. The diagonal entries are $a_1
		= 1$, $a_2 = 1 - \varphi(p^{m+1})$, and $a_i = -\varphi(p^{m-1+i})$ for
	$i \geq 3$.

	The determinant of the $2 \times 2$ block is given by:\[
		\det(M) = a_1 a_2 - b_1^2 = 1(1 - \varphi(p^{m+1})) - 1^2 =
		-\varphi(p^{m+1}).
	\]
	Since $\varphi(p^{m+1}) \neq 0$ and all remaining diagonal entries $a_i$ are
	strictly negative, the total determinant $\det(X) = \det(M) \prod_{i=3}^m a_i$
	is nonzero. Thus, $\phi_C(-1) \neq 0$, and $-1$ is not an
	eigenvalue.
\end{proof}

\bigskip

\section*{Acknowledgments}
This research project was supported by King Mongkut's University of Technology
Thonburi: KMUTT Partnering Initiative Grant fiscal year 2025 under KIRIM number
30000117.

\medskip

This research project is supported by the Faculty of Science, Mahidol
University.

\bibliographystyle{unsrt}
\nocite{*}
\bibliography{refs}

@article{Anderson1999On,
  author  = {Anderson, D. F. and Livingston, P. S.},
  title   = {The zero-divisor graph of a commutative ring},
  journal = {Journal of Algebra},
  volume  = {217},
  pages   = {434--447},
  year    = {1999}
}

@article{Aronstein1997Fractional,
  author  = {Aronstein, D. L. and Stroud, C. R.},
  title   = {Fractional wave-function revivals in the infinite square well},
  journal = {Physical Review A},
  volume  = {55},
  number  = {6},
  pages   = {4526--4537},
  year    = {1997},
}

@article{Bachman2012Perfect,
  author  = {Bachman, R. and Fredette, E. and Fuller, J. and Landry, M. and Opperman, M. and Tamon, C. and Tollefson, A.},
  title   = {Perfect state transfer on quotient graphs},
  journal = {Quantum Information and Computation},
  volume  = {12},
  number  = {3--4},
  pages   = {293--313},
  year    = {2012},
}

@article{Bajaj2022On,
  author  = {Bajaj, S. and Panigrahi, P.},
  title   = {On the adjacency spectrum of zero divisor graph of ring {$\mathbb{Z}_n$}},
  journal = {Linear Algebra and Its Application},
  volume  = {21},
  number  = {10},
  year    = {2022}
}

@article{Basic2009Perfect,
  author  = {Ba{\v{s}}i{\'c}, M. and Petkovi{\'c}, M. D. and Stevanovi{\'c}, D.},
  title   = {Perfect state transfer in integral circulant graphs},
  journal = {Applied Mathematics Letters},
  volume  = {22},
  number  = {7},
  pages   = {1117--1121},
  year    = {2009},
}

@article{Basic2010Perfect,
  author  = {Ba{\v{s}}i{\'c}, M. and Petkovi{\'c}, M. D.},
  title   = {Perfect state transfer in integral circulant graphs of non-square-free order},
  journal = {Linear Algebra and its Applications},
  volume  = {433},
  number  = {1},
  pages   = {149--163},
  year    = {2010}
}

@article{Bernard2018Graph,
  author  = {Bernard, P. A. and Chan, A. and Loranger, {\'E}. and Tamon, C. and Vinet, L.},
  title   = {A graph with fractional revival},
  journal = {Physics Letters A},
  volume  = {382},
  number  = {5},
  pages   = {259--264},
  year    = {2018},
}

@article{Bose2003Quantum,
  author  = {Bose, S.},
  title   = {Quantum communication through an unmodulated spin chain},
  journal = {Physical Review Letters},
  volume  = {91},
  number  = {20},
  pages   = {207901},
  year    = {2003},
  doi     = {10.1103/PhysRevLett.91.207901}
}

@article{Cao2020Perfect,
  author  = {Cao, X. and Chen, B. and Ling, S.},
  title   = {Perfect state transfer on {C}ayley graphs over dihedral groups: {T}he non-normal case},
  journal = {The Electronic Journal of Combinatorics},
  volume  = {27},
  number  = {2},
  pages   = {P2.28},
  year    = {2020},
}

@article{Cao2021Perfect,
  author  = {Cao, X. and Feng, K. and Tan, Y. Y.},
  title   = {Perfect state transfer on weighted {A}belian {C}ayley graphs},
  journal = {Chinese Annals of Mathematics, Series B},
  volume  = {42},
  number  = {4},
  pages   = {625--642},
  year    = {2021},
}

@article{Chan2019Quantum,
  author    = {Chan, A. and Coutinho, G. and Tamon, C. and Vinet, L. and Zhan, H.},
  title     = {Quantum fractional revival on graphs},
  journal   = {Discrete Applied Mathematics},
  volume    = {269},
  pages     = {86--98},
  year      = {2019},
  publisher = {Elsevier}
}

@article{Chan2020Fractional,
  author  = {Chan, A. and Coutinho, G. and Tamon, C. and Vinet, L. and Zhan, H.},
  title   = {Fractional revival and association schemes},
  journal = {Discrete Math},
  volume  = {343},
  number  = {11},
  pages   = {112018},
  year    = {2020}
}

@article{Chan2022Fundamentals,
  author  = {Chan, A. and Coutinho, G. and Drazen, W. and Eisenberg, O. and Godsil, C. and Kempton, M. and Lippner, G. and Tamon, C. and Zhan, H.},
  title   = {Fundamentals of fractional revival in graphs},
  journal = {Linear Algebra and its Applications},
  volume  = {655},
  pages   = {129--158},
  year    = {2022}
}

@article{Chattopadhyay2020Laplacian,
  author  = {Chattopadhyay, S. and Patra, K. L. and Sahoo, B. K.},
  title   = {Laplacian eigenvalues of the zero divisor graph of the ring {$\mathbb{Z}_n$}},
  journal = {Linear Algebra and its Applications},
  volume  = {584},
  pages = {267-286},
  year    = {2020}
}

@article{Cheung2011Perfect,
  author  = {Cheung, W. C. and Godsil, C.},
  title   = {Perfect state transfer in cubelike graphs},
  journal = {Linear Algebra and its Applications},
  volume  = {435},
  number  = {10},
  pages   = {2468--2474},
  year    = {2011}
}

@article{Christandl2004Perfect,
  author  = {Christandl, M. and Datta, N. and Ekert, A. and Landahl, A. J.},
  title   = {Perfect state transfer in quantum spin networks},
  journal = {Phys. Rev. Lett.},
  volume  = {92},
  number  = {18},
  pages   = {187902},
  year    = {2004}
}

@article{Dooley2014Fractional,
  author  = {Dooley, S. and Spiller, T. P.},
  title   = {Fractional revivals, multiple-{S}chr{\"o}dinger-cat states, and quantum carpets in the interaction of a qubit with {N} qubits},
  journal = {Phys. Rev. A.},
  volume  = {90},
  number  = {1},
  pages   = {012320},
  year    = {2014}
}

@article{Fan2013Pretty,
  author  = {Fan, X. and Godsil, C.},
  title   = {Pretty good state transfer on double stars},
  journal = {Linear Algebra Appl.},
  volume  = {438},
  pages   = {2346--2358},
  year    = {2013}
}

@article{Farhi1998Quantum,
  author  = {Farhi, E. and Gutmann, S.},
  title   = {Quantum computation and decision trees},
  journal = {Phys. Rev. A.},
  volume  = {58},
  number  = {2},
  pages   = {915},
  year    = {1998}
}

@article{Genest2016Quantum,
  author  = {Genest, V. X. and Vinet, L. and Zhedanov, A.},
  title   = {Quantum spin chains with fractional revival},
  journal = {Ann. Phys.},
  volume  = {371},
  pages   = {348--367},
  year    = {2016}
}

@article{Godsil2012State,
  author  = {Godsil, C.},
  title   = {State transfer on graphs},
  journal = {Discrete Math.},
  volume  = {312},
  number  = {1},
  pages   = {129--147},
  year    = {2012}
}

@misc{Godsil2017State,
  author        = {Godsil, C.},
  title         = {State transfer on graphs},
  year          = {2017},
  eprint        = {1102.4898v3},
  archivePrefix = {arXiv},
  primaryClass  = {math.CO}
}

@article{Godsil2022Fractional,
  author  = {Godsil, C. and Zhang, X.},
  title   = {Fractional revival on non-cospectral vertices},
  journal = {Linear Algebra and its Applications},
  volume  = {654},
  pages   = {69--88},
  year    = {2022}
}

@inproceedings{Grover1996Fast,
  author    = {Grover, L. K.},
  title     = {A fast quantum mechanical algorithm for database search},
  booktitle = {Proceedings of the Twenty-Eighth Annual ACM Symposium on Theory of Computing},
  year      = {1996},
  pages     = {212--219}
}

@article{Jitngam2026Quantum,
  author  = {Jitngam, S. and Kumam, P. and Sriwongsa, S.},
  title   = {Quantum fractional revival on unitary {C}ayley graphs over finite commutative rings},
  journal = {Linear Multilinear Algebra},
  pages   = {1--17},
  year    = {2026}
}

@article{Kendon2011Perfect,
  author  = {Kendon, V. M. and Tamon, C.},
  title   = {Perfect state transfer in quantum walks on graphs},
  journal = {Journal of Computational and Theoretical Nanoscience},
  volume  = {8},
  number  = {3},
  pages   = {422--433},
  year    = {2011},
}

@article{Meemark2014Perfect,
  author  = {Meemark, Y. and Sriwongsa, S.},
  title   = {Perfect state transfer in unitary {C}ayley graphs over local rings},
  journal = {Transactions on Combinatorics},
  volume  = {3},
  number  = {4},
  pages   = {43--54},
  year    = {2014},
}

@article{Monterde2022Strong,
  author  = {Monterde, H.},
  title   = {Strong cospectrality and twin vertices in weighted graphs},
  journal = {The Electronic Journal of Linear Algebra},
  volume  = {38},
  pages   = {494-518},
  year    = {2021}
}

@article{Monterde2023Fractional,
  author  = {Monterde, H.},
  title   = {Fractional revival between twin vertices},
  journal = {Linear Algebra and its Applications},
  volume  = {676},
  pages   = {25--43},
  year    = {2023}
}

@article{Rohith2015Visualizing,
  author  = {Rohith, M. and Sudheesh, C.},
  title   = {Visualizing revivals and fractional revivals in a {Kerr} medium using an optical tomogram},
  journal = {Physical Review A},
  volume  = {92},
  number  = {5},
  pages   = {053828},
  year    = {2015},
}

@article{Soni2025Quantum,
  author  = {Soni, R. and Choudhary, N. and Singh, N. P.},
  title   = {Quantum fractional revival in unitary {C}ayley graphs},
  journal = {Lobachevskii Journal of Mathematics},
  volume  = {46},
  number  = {4},
  pages   = {1922--1928},
  year    = {2025},
}

@article{Thongsomnuk2019Perfect,
  author  = {Thongsomnuk, I. and Meemark, Y.},
  title   = {Perfect state transfer in unitary {C}ayley graphs and gcd-graphs},
  journal = {Linear Multilinear Algebra},
  volume  = {67},
  number  = {1},
  pages   = {39--50},
  year    = {2019}
}

@article{Wang2024Fractional,
  author  = {Wang, J. and Wang, L. and Liu, X.},
  title   = {Fractional revival on {C}ayley graphs over {A}belian groups},
  journal = {Discrete Math},
  volume  = {347},
  number  = {12},
  pages   = {114218},
  year    = {2024}
}

@article{Wang2025Fractional,
  author  = {Wang, J. and Wang, L. and Liu, X.},
  title   = {Fractional revival on semi-{C}ayley graphs over {A}belian groups},
  journal = {Linear Multilinear Algebra.},
  volume  = {73},
  number  = {8},
  pages   = {1611--1633},
  year    = {2025}
}
\end{document}